\documentclass[11pt,a4paper]{article}
\usepackage[T1]{fontenc}
\usepackage[utf8]{inputenc}
\usepackage{lmodern}
\usepackage[left=2.5cm,right=2.5cm,top=2.0cm,bottom=2.5cm]{geometry}
\usepackage{microtype}
\usepackage{amsmath,amssymb,amsfonts,amsthm,mathtools,mathrsfs}
\usepackage{graphicx}
\usepackage{subcaption}
\usepackage{booktabs}
\usepackage{multirow}
\usepackage{array,makecell}
\usepackage{siunitx}
\usepackage{float}
\usepackage{adjustbox}
\usepackage{tikz}
\usepackage{tikz-cd}
\usepackage{pgfplots}
\usepackage{extarrows}
\usepackage{url}
\usepackage[colorlinks=true,linkcolor=blue,citecolor=blue,urlcolor=blue]{hyperref}
\usepackage[nameinlink,noabbrev]{cleveref}
\usepackage[section]{placeins}
\usetikzlibrary{matrix,positioning,quotes,shapes,arrows,arrows.meta,decorations.pathreplacing}
\pgfplotsset{compat=1.17}
\newcommand{\egyk}{\frac{1}{2}}
\newcommand{\er}{\mathbb{R}}
\newcommand{\xx}{\mathbf{x}}
\newcommand{\ff}{\mathbf{f}}
\newcommand{\uu}{\mathbf{u}}
\newcommand{\bb}{\mathbf{b}}
\newcommand{\relu}{\mathrm{ReLU}\:}
\newcommand{\NNN}{\mathrm{NN}}
\newcommand{\TV}{\mathrm{TV}}
\newcommand{\lin}{\mathrm{lin}\:}

\newtheorem{theorem}{Theorem}
\newtheorem{prop}[theorem]{Proposition}

\title{Neural Network-Assisted Refinement of Traditional Schemes for\\One-Dimensional Scalar Conservation Laws}
\author{
Imre Fekete\textsuperscript{1,2}
\and Ferenc Izsák\textsuperscript{3}
\and Vendel P. Kupás\textsuperscript{1}}
\date{}

\begin{document}
\maketitle
\begin{center}
\small
\textsuperscript{1} Department of Applied Analysis and Computational Mathematics, ELTE Eötvös Loránd University, Budapest, Hungary\\
\textsuperscript{2} Department of Network and Data Science, Central European University, Vienna, Austria\\
\textsuperscript{3} Department of Applied Analysis and Computational Mathematics and NumNet HUN-REN--ELTE Research Group, ELTE Eötvös Loránd University, Budapest, Hungary
\end{center}

\begin{abstract}
A clear link is established between conventional numerical methods and neural network approximations for solving one-dimensional scalar conservation laws. The focus is on the construction of an appropriate flux term in the case of convex flux functions for improving the classical schemes. The first neural network developed here is able to rediscover Godunov's method, while the second one emulates the behavior of a second-order slope-limiter function. In this way, by merging them, second-order reconstruction-based schemes can be developed. The networks presented here employ a minimal number of parameters, significantly reducing the complexity compared to previous approaches. These networks can also be linked consecutively to get a deep one corresponding to multiple time steps. Training them with an appropriate loss leads to stable schemes, improving even the classical methods without increasing their complexity.
\end{abstract}


\section{Introduction}
\label{intro}

Neural networks are now widely used across all areas of scientific computing. Significant efforts have also been made to incorporate them into, or even replace, conventional numerical partial differential equation (PDE) solvers. In general terms, these approaches fall into two categories.

In the first category, a traditional numerical method is entirely replaced by a suitably designed neural network, which directly computes the full solution. Among such methods, the most widely adopted are physics-informed neural networks (PINNs), which have been applied to a variety of PDEs \cite{raissi19}. These networks typically rely on pointwise function evaluations and often involve dense architectures, resulting in an excessively large number of parameters.
Whenever a number of related approaches were proposed for the numerical solution of various PDEs, we do not list them as they are far from our focus. A related study discussing
and comparing several PINN-type approaches for Burgers equation is available in \cite{savovic23}.

 A more promising direction is offered by neural operators, particularly Fourier neural operators (FNOs), which have enabled competitive PDE simulations \cite{kovachki23}. At the same time, this approach makes use of global basis functions so that we cannot expect networks with local connections.

The second category of neural network-based approaches aims to enhance or partially replace components of existing numerical methods. In the context of computational fluid dynamics, we highlight a few representative examples below.

\begin{itemize}
    \item [-]
    Neural networks have been employed in \cite{ruggeri22} and further explored in \cite{xu25} to construct approximate Riemann solvers. However, using dense layers for this purpose can significantly increase the number of parameters.

\item [-]
The combination of lower- and higher-order spatial approximations is typically achieved using flux limiters. Since this process involves heuristic steps, neural network-based optimization methods  \cite{hillebrand23}, \cite{nguyen22} can outperform traditional analytic techniques.

 \item [-]
 Additionally, neural networks can be utilized to estimate local smoothness and determine optimal combinations of lower-order reconstructions in the so-called weighted essentially non-oscillatory (WENO) spatial discretizations, as proposed in \cite{kossaczka24} and \cite{nogueira24}.

\item [-]
Neural networks can also be used to develop numerical fluxes
 for solving conservation laws. Based on this, the authors in \cite{chen24} develop conservative schemes, which may contain several time steps. This can also be done without any knowledge of the analytic flux function. Our approach is closely related to this work.

 \end{itemize}

Nevertheless, a common weakness of many related approaches is their lack of emphasis on minimizing the number of network parameters. In other words, the structure of the underlying neural network is often not optimized in a principled way. Note that the minimization of the number of parameters in certain families of neural networks is also of high interest, as this can set the limit of the related numerical methods. For an overview of the architectures of the related neural networks, see Table \ref{compare_no_par} in Section \ref{discussion}.

In this paper, we develop neural network-based numerical methods for
{ one-dimensional scalar conservation laws,} with these considerations in mind as well.
More specifically, while following the core principles of conventional finite-volume methods for solving conservation laws, our goal is to design neural networks of minimal complexity that can effectively or even completely replicate their behavior. {By optimizing the associated weights, i.e., training the associated network, we can further fine-tune and enhance the accuracy and efficiency of the associated numerical methods.} For this, we utilize a deep neural network mimicking simultaneously a number of time steps in a numerical method. In this way, using appropriate loss functions, which is also a core idea of our approach, we can enforce the stability of the proposed method.

As we point out in the numerical experiments, our approach - regarding the minimal number of parameters and the exact replication of conventional solvers - has several benefits:
\begin{itemize}
    \item
we can prevent overfitting effects while training the corresponding
neural network,
    \item
    we will be able to train the method over several time steps
    controlling its stability,
    \item
    we can minimize the computational complexity of the optimization process,
    \item
    we can make a clear link between the conventional Godunov and flux limiter type approaches and neural network-based approaches.
\end{itemize}

\subsection{Mathematical preliminaries}
Since we establish a clear link between the numerical methods of scalar conservation laws and neural networks, we shortly summarize the corresponding
notions and concepts.
\subsubsection{Numerical methods for scalar conservation laws}\label{sec_111}
The mathematical model of one-dimensional scalar conservation laws is given with the initial-boundary value problem
\begin{equation}\label{cons_law}
\begin{cases}
     \partial_t u(t,x) + \nabla\cdot f(u) (t,x) = 0\quad t\in (0,T),\; x\in\Omega\\
    u(0,x) = u_0(x) \quad x\in\Omega,
\end{cases}
\end{equation}
on the given { bounded domain
$\Omega\subset\er$},
which should be equipped with appropriate boundary conditions \cite{bardos79}. Here the function $u:(0,T)\times\Omega\to\er$ corresponds to the unknown quantity with a given initial state $u_0$.\\
Applying a spatial grid and the time step $\delta$, a conventional conservative finite volume-based numerical scheme reads as
\begin{equation}\label{cons_form}
u^{n+1}_j = u^n_j -
\frac{\delta}{|I_j|}\cdot
(\hat f_{j+\frac{1}{2}} - \hat f_{j-\frac{1}{2}}),
\end{equation}
where for all possible index $j$,
\begin{itemize}
\item
$|I_j|$ denotes the length of the $j$th grid cell $I_j$,
\item
$u^n_j$ denotes the cell average of $u$ at $n\delta$ in  $I_j$,
\item
$\hat f_{j+\frac{1}{2}}$ is an approximation of the time average of $f(u)$
at the interface of $I_j$ and $I_{j+1}$ over the time interval
$(n\delta, (n+1)\delta)$.
\end{itemize}
First we investigate the case when $\hat f_{j+\frac{1}{2}}:\er^2\to\er$ depends on
$u^n_{j}$ and $u^n_{j+1}$.\\
The real challenge in \eqref{cons_form} is to develop an
accurate numerical flux $\hat f_{j+\frac{1}{2}}$ ensuring stability.
Usually, this property is measured by the total variation (TV) (semi)norm, see \cite{leveque92}.
Therefore, starting with the classical work \cite{harten83}, many efforts were made to construct total variation diminishing (TVD) schemes.

The complexity of this problem has long been recognized. According to the classical  result in \cite{godunov59},
conservative TVD schemes with second-order spatial convergence cannot be linear even for linear equations. This key observation gives rise to apply possibly nonlinear neural networks for constructing numerical fluxes. This is based on one of the following principles:
\begin{itemize}
    \item
    One can switch between lower and higher-order approximations of the real flux. { In many cases, this procedure leads to the limitation of the numerical fluxes, a technique known as flux limiting}
    \cite{harten83}.
    \item
    Alternatively, based on an estimate
    of local smoothness, multiple flux approximations can be weighted adaptively to avoid oscillations. This principle leads to the WENO methods \cite{jiang96}, \cite{liu94}.
\end{itemize}

The ancestor of all modern methods is the so-called Godunov's scheme, which in the one-dimensional case both for {any function  $f\in C^1(\er)$} reads as
\begin{equation}\label{G_scheme}
    \hat f_{j+\frac{1}{2}} =
    \begin{cases}
       \max \limits_{u\in [u^n_{j+1},u^n_{j}]} f(u) \qquad \textrm{if}\;
       u^n_{j+1}\le u^n_{j}\\
       \min\limits_{u\in [u^n_{j},u^n_{j+1}]} f(u)\qquad
       \textrm{if}\;
       u^n_{j}\le u^n_{j+1}.
    \end{cases}
\end{equation}
{\emph{Remark:}
For strictly convex or concave functions, \eqref{G_scheme} can be rewritten in even more
concrete terms. This will be done if we rewrite
the flux as the output of a neural network.}\medskip

To extend this baseline numerical method and later on, the corresponding neural network, we use reconstruction schemes with flux limiters. To simplify the notations, the upper index $n$ will be dropped in the consecutive analysis.

The basic idea here is that keeping the local average $u_j$ on $I_j$, we also
define separate approximations on the right and left end of this interval.\\
In any case, for an accurate reconstruction on $I_j$, we need (at least), the values $u_{j-1}, u_{j}$ and $u_{j+1}$.
Computing an approximation of the local slope $s_*$, the reconstructed values
\begin{equation}\label{final_rec_step}
    u_j^L = u_j - \frac{h\cdot s_*}{2}\quad\textrm{and}\quad u_j^R = u_j + \frac{h\cdot s_*}{2}
\end{equation}
will be given on the left and right end of $I_j$, respectively.

In the framework of the conventional approach, the deviation $h\cdot s_*$ above is determined as
\begin{equation}\label{slope_basic}
  s := h\cdot s_* = (a+b)\cdot \Phi\left(\frac{a}{a+b}\right) :=
   (a+b)\cdot \Phi(r),
\end{equation}
with the local differences $a=u_{j}-u_{j-1}$ and $b=u_{j+1}-u_j$. {Here, instead of the conventional choice $\frac{a}{b}$, we use the ratio $r=\frac{a}{a+b}$ as proposed in} \cite{berger05}. {In this way, the corresponding (limiter) function $\Phi$ is supported on $[0,1]$, which greatly simplifies its optimization within our neural network.}

{ To satisfy the total variation diminishing (TVD) property, \eqref{slope_basic} must vanish whenever
$ab<0$. Consequently, $\Phi(r)=0$ in $\er\setminus [0,1]$. Since many widely used TVD limiter functions are piecewise linear, we restrict our attention to this class of limiter functions. Accordingly, for a fixed $K\in\mathbb{N}$, we assume}
\begin{equation}\label{phi_def}
    \Phi(r) = m_k(r-x_k) + \Phi(x_k) \quad x_k\le r\le x_{k+1}, \; k=0,1,\dots, K,
\end{equation}
where { $0=x_0<x_1<\ldots<x_K=1$ and $m_0, m_1, \dots, m_{K-1}$  denote the breakpoints} and the consecutive slopes, respectively; see Figure \ref{pw_lin}.\\
\begin{figure}[ht]
\centering
\begin{tikzpicture}
  \begin{axis}[
      width=13cm,
      height=6cm,
      axis lines=middle,
      xlabel={$x$},
      ylabel={$\Phi(x)$},
      xmin=-0.2, xmax=1.2,
      ymin=-0.2, ymax=1.2,
      xtick={0,0.25,0.5,0.75,1.0},
      xticklabels={$x_0$,$x_1$,$x_2$,$x_3$,$x_4=1$},
      ytick=\empty,
      domain=-0.2:1.2,
      samples=100,
      thick,
      clip=false,
      grid=major,
    ]
    \addplot[
      color=blue,
      mark=*,
    ] coordinates {
      (0.0, 0.0)
      (0.25, 0.5)
      (0.5, 0.3)
      (0.75, 0.9)
      (1.0, 0.0)
    };
       \addplot[
      color=blue,
      mark=,
    ] coordinates {
      (-0.2, 0)
      (0.0, 0.0)
      (0.25, 0.5)
      (0.5, 0.3)
      (0.75, 0.9)
      (1.0, 0.0)
      (1.2, 0)
    };
    \node at (axis cs:0.0,0.0) [below left] {$x_0$};
    \node at (axis cs:0.25,0.5) [above right] {$\Phi(x_1)$};
    \node at (axis cs:0.5,0.3) [below right] {$\Phi(x_2)$};
    \node at (axis cs:0.75,0.9) [above right] {$\Phi(x_3)$};
    \node at (axis cs:0.125,0.3) [above] {$m_0$};
    \node at (axis cs:0.375,0.35) [below] {$m_1$};
    \node at (axis cs:0.625,0.4) [above] {$m_2$};
    \node at (axis cs:0.875,0.65) [below] {$m_3$};
  \end{axis}
\end{tikzpicture}
\caption{The piecewise linear, continuous function $\Phi$ with breakpoints $x_0=0 < x_1 < x_2 < x_3 < x_4=1$. Function values $\Phi(x_i)$ are shown at each node, with slopes $m_k$ on each interval.}\label{pw_lin}
\end{figure}
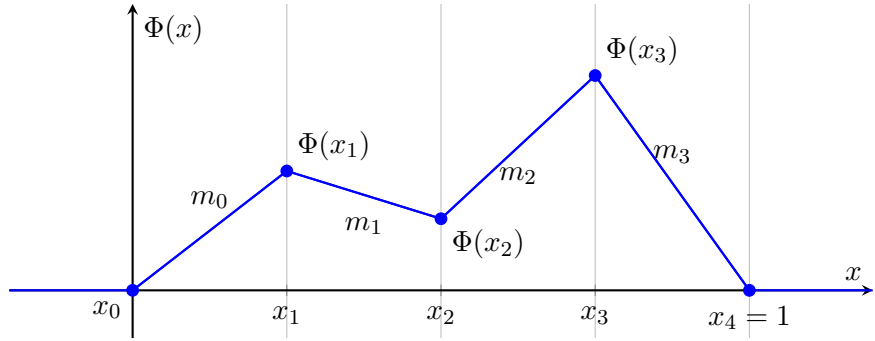

\subsubsection{Neural networks in short}

In precise terms, artificial neural networks are directed graphs with certain groups of vertices such that we assign functions to their edges.

From the practical point of view, we identify neural networks with a certain kind of function composition $\mathcal{F}$ and with this, the
network is given as
\begin{equation}\label{NN}
\mathcal{F}:
\xx_0\xmapsto{g_1} \xx_1
\xmapsto{g_2} \xx_2
\xmapsto{g_3} \cdots
\xmapsto{g_N} \xx_N.
\end{equation}
Here the stages called the layers with the values
$\xx_0\in\er^{d_0}, \xx_1\in\er^{d_1}, \dots, \xx_N\in\er^{d_N}$.

In the classical setup, the function $g_j$
is given with
$$
g_j (\xx_{j-1})= \boldsymbol{\sigma_j}(A_j\xx_{j-1} + \bb_{j-1}),
$$
where $A_j:\er^{d_{j-1}}\to\er^{d_{j}} $ is linear, $\bb_j\in\er^{d_{j}}$ and  the so-called activation function $\boldsymbol{\sigma_j}: \er^{d_{j}}\to\er^{d_{j}}$ is given
with
$$
\boldsymbol{\sigma_j}(y_1, y_2, \dots, y_{d_j})^T =(\sigma_j(y_1),  \sigma_j(y_2), \dots, \sigma_j(y_{d_j}))^T,
$$
which together with the dimensions are fixed.
The entries of $A_1, A_2,\dots, A_N$ and
$b_1, b_2,\dots, b_N$ are called the parameters of the network. These should be optimized { so that $\mathcal{F}$ approximates the target function, which in our case, is} the solution operator $F: u(0,\cdot)\mapsto u(T,\cdot)$.

At the same time, we will complete the conventional network in \eqref{NN} with
layers corresponding to given functions, which do not involve tunable parameters. This will only very slightly increase the complexity of the
optimization problem for the network parameters.

\section{Results}
\label{sec_res}
Before introducing our neural network, we
lay down the main principles and the motivations behind it.
\subsection{Principles and main ideas}
\label{subsec_ideas}
The proposed neural network will correspond to the solution map
$u(0,\cdot)\to u(T,\cdot)$ of the conservation law in \eqref{cons_law}. Accordingly, the input layer contains the
interval averages $\uu^0 = (\dots, u_j^0, u_{j+1}^0, \dots)$ at the initial time $t=0$, while the output layer
$\uu^N = (\dots, u_j^N, u_{j+1}^N, \dots)$ approximates the averages at the final time $T = N\cdot\delta$.

\begin{itemize}
    \item
To mimic the conservative form \eqref{cons_form} of a general scheme, we should employ
a convolutional layer to take into account a stencil of the neighboring elements.
This idea remains also valid in the multidimensional case.

\item
The activation functions can make the scheme to be nonlinear, which is
a basic requirement to have higher-order accuracy.

\item
After constructing a few layers corresponding to a single time step,
we repeat them consecutively leading to a deep network. This has also several
advantages.

\begin{itemize}
    \item
    We can optimize the time stepping under the same hood regarding its accuracy and stability.
    \item
    We can keep the number of parameters at a low level, since in the consecutive blocks, the same ones will be used.
\end{itemize}

\item
The possible increment of the total variation can be incorporated into the
loss function and can be penalized using an appropriate weight.

\end{itemize}

\subsection{The neural networks corresponding to numerical methods}

\subsubsection{Construction of basic fluxes}\label{sec221}
We first discuss the neural network form of the Godunov method.
{We stress that, within this framework, the limiter can be obtained without any prior knowledge or analytical derivation. Instead, its parameters are determined through a standard optimization procedure, or, in the context of neural networks, through a learning procedure.}

For the most simple neural network setup, we use the following assumptions:
\begin{itemize}
    \item [(A1)]
    The function $f:\er\to\er$ is strictly convex.
     \item [(A2)]
    We have  $f(x)\ge 0$ with $f(0) = 0$.
    \item [(A3)]
    The function $f$ is even, i.e. $f(x) = f(-x)$.
\end{itemize}

Some traffic flow models assume $f(u)\ge 0$ since here $u$ denotes the vehicle density and the flux $f(u)=u\cdot v(u)$ represents the flow rate, which is non-negative in the absence of contraflow traffic.
On the other hand, here, with a zero density, really no flow rate is generated.

For the most general case, we use the notation $u^{\min}\in\er\cup-\infty\cup\infty$
for the minimum site of the strictly convex function $f$ and assume the following.
\begin{itemize}
    \item [(A11)]
  $u^{\min}\in \er$
\end{itemize}
The Godunov flux can be then expressed in an alternative form, which is related to neural networks, as stated in the following.

\begin{prop}\label{basic_prop_strict}
Using only assumptions (A1) and (A11),  the Godunov flux in \eqref{G_scheme} can be given as the function
$$
\begin{aligned}
&\hat{f}_{\NNN}(u^n_j,u^n_{j+1})\\
&=
\max\left\{f\left(\relu(u^{n}_j-u^{\min})+u^{\min}\right);
f\left(-\relu(u^{\min}-u^{n}_{j+1})+u^{\min}\right)\right\}.
\end{aligned}
$$
\end{prop}

\textbf{Proof.} We verify the statement by means of a case analysis. In each case, we use \eqref{G_scheme} and the definition of $\hat{f}_{\NNN}$.

For  $u^{\min}<u^{n}_j<u^{n}_{j+1}$,  the function $f$ is strictly monotone increasing in
$[u^{\min}, \infty)$ such that
$$
\hat{f}_{j+\egyk}(u^{n}_j,u^{n}_{j+1}) =  \min\limits_{u\in [u^n_{j},u^n_{j+1}]} f(u)
= f(u^{n}_j)
$$
and also, by definition,
$$
\hat{f}_{\NNN}(u^{n}_j,u^{n}_{j+1}) = \max\left\{ f(u^{n}_j);f(u^{\min})\right\}= f(u^{n}_j).
$$

{ For $u^{\min}<u^{n}_{j+1}<u^{n}_j$,
we simply obtain
$$
\hat{f}_{j+\egyk}(u^{n}_j,u^{n}_{j+1}) =  \max\limits_{u\in [u^n_{j+1},u^n_{j}]} f(u) = f(u^{n}_j)
 = \max\left\{ f(u^{n}_j);f(u^{\min})\right\} = \hat{f}_{\NNN}(u^{n}_j,u^{n}_{j+1}).
$$
In a similar way, for $u^{n}_j<u^{\min}<u^{n}_{j+1}$, we get
$$
\hat{f}_{j+\egyk}(u^{n}_j,u^{n}_{j+1}) =  \min\limits_{u\in [u^n_{j},u^n_{j+1}]} f(u) = f(u^{\min})
= \max\left\{ f(u^{\min});f(u^{\min})\right\}=
\hat{f}_{\NNN}(u^{n}_j,u^{n}_{j+1})
$$
and also, for $u^{n}_{j+1}<u^{\min}<u^{n}_j$, we obtain
$$
\hat{f}_{j+\egyk}(u^{n}_j,u^{n}_{j+1}) = \max\left\{ f(u^{n}_j);f(u^{n}_{j+1})\right\} =
\hat{f}_{\NNN}(u^{n}_j,u^{n}_{j+1}).
$$
For $u^{n}_j<u^{n}_{j+1}<u^{\min}$, $f$ is strictly monotone decreasing in
$(-\infty, u^{\min}]$, such that
$$
\hat{f}_{j+\egyk}(u^{n}_j,u^{n}_{j+1}) =  \min\limits_{u\in [u^n_{j},u^n_{j+1}]} f(u)
= f(u^{n}_{j+1}) = \max\left\{ f(u^{\min});f(u^{n}_{j+1})\right\} =
\hat{f}_{\NNN}(u^{n}_j,u^{n}_{j+1}).
$$

Finally, in case of $u^{n}_{j+1}<u^{n}_j<u^{\min}$,  $f$ is again strictly monotone decreasing in $(-\infty, u^{\min}]$ such that
$$
\hat{f}_{j+\egyk}(u^{n}_j,u^{n}_{j+1}) = \max\limits_{u\in [u^n_{j+1},u^n_{j}]} f(u) = f(u^{n}_{j+1}) = \max\left\{ f(u^{\min});f(u^{n}_{j+1})\right\}=
{f}_{\NNN}(u^{n}_j,u^{n}_{j+1}).
$$}
In this way, in all cases, we have verified the equality
$$
\hat{f}_{j+\egyk}(u^{n}_j,u^{n}_{j+1}) = {f}_{\NNN}(u^{n}_j,u^{n}_{j+1}).
\qquad \square
$$

Alternatively, we may use more
conditions to simplify the neural network-related form $f_{\NNN}$ of the Godunov flux in Proposition \ref{basic_prop_strict}.
\begin{prop}\label{basic_prop}
Using the assumptions (A1)-(A3), the Godunov flux in \eqref{G_scheme} can be given as the function
    \begin{equation}\label{G_NN1}
 \begin{pmatrix}
        u_j^n\\u_{j+1}^n
    \end{pmatrix}
    \xlongrightarrow{A \cdot }
    \begin{pmatrix}
     {\;}\\{\;}
    \end{pmatrix}
     \xlongrightarrow{ \textrm{ReLu}}
    \begin{pmatrix}
    {\;}\\{\;}
    \end{pmatrix}
     \xlongrightarrow{f}
    \begin{pmatrix}
    {\;}\\{\;}
    \end{pmatrix}
     \xlongrightarrow{\max}
     \hat f_{j+\frac{1}{2}},
    \end{equation}

    where $A = \begin{pmatrix}
        1 &0\\ 0&-1
    \end{pmatrix}$.

\end{prop}
\emph{Proof:}
First we note that according to (A1)-(A2), the function $f$ is monotone decreasing in $\er^-$ and  monotone increasing in $\er^+$.\\
We rewrite the Godunov's scheme in
the quadrants determined by $(u_j^n, u_{j+1}^n)$. \\
Applying the first two components in \eqref{G_NN1}, we get
$ \begin{pmatrix}
      \textrm{ReLu}(u_j^n)\\
      \textrm{ReLu}(-u_{j+1}^n)
    \end{pmatrix} $,
which has the value
$ \begin{pmatrix}
      u_j^n\\
      0
    \end{pmatrix} $,
    $ \begin{pmatrix}
      0\\0
    \end{pmatrix} $,
    $ \begin{pmatrix}
      0\\
      -u_{j+1}^n
    \end{pmatrix} $ and
    $ \begin{pmatrix}
      u^n_j\\- u_{j+1}^n
    \end{pmatrix} $ in the first, second, third and fourth quadrant, respectively. {The maximum of these values will be related with the Godunov flux in \eqref{G_scheme}.}


\begin{itemize}
    \item
    In the first quadrant, \eqref{G_NN1} gives
    finally {
   $$
   \max\{0, f(u_j^n)\} = f(u_j^n) = \begin{cases}
    \displaystyle{\min_{[u_j^n, u_{j+1}^n]} f   \quad 0\le u_j^n\le u_{j+1}^n},\\
    \displaystyle{\max_{[u_{j+1}^n, u_{j}^n]} f \quad u_j^n\ge u_{j+1}^n}.
   \end{cases}
   $$
 }
\item
In the second quadrant, { $0\in [u_j^n, u_{j+1}^n]$,  so that (A2)
implies
$$
\max\{0, 0\} = 0 = \min_{[u_j^n, u_{j+1}^n]} f.$$
}

\item
    In the third quadrant, { in case of \eqref{G_NN1},  using also (A3), we have
    $$
    \max\{0, f(-u_{j+1}^n)\} = f(-u_{j+1}^n) =
    f(u_{j+1}^n) =
    \begin{cases}
     \displaystyle{\min_{[u_j^n, u_{j+1}^n]}f\quad    u_j^n\le u_{j+1}^n\le 0},\\
    \displaystyle{\max_{[u_{j+1}^n, u_{j}^n]} f\quad u_{j+1}^n\le u_{j}^n\le 0.}
    \end{cases}
    $$
}
\item
In the fourth quadrant, where $u_j^n\ge 0\ge u_{j+1}^n$, \eqref{G_NN1} gives
{
$$
\max\{f(u_j^n), f(-u_{j+1}^n)\} = \max\{f(u_j^n), f(u_{j+1}^n)\} = \max_{[u_{j+1}^n, u_{j}^n]} f.
$$
}
\end{itemize}
Summarized, in all cases, we have verified the equality in the statement.
\quad $\square$\medskip

We first note that the matrix multiplications in \eqref{G_NN1}
can also be recognized as the effect of a convolution layer. In concrete terms, the rows of the matrix $A_1$ and $A_2$ can be identified with filters, giving the two components of $A\cdot \begin{pmatrix}
    u^n_j\\ u^n_{j+1}
\end{pmatrix}$ in \eqref{G_NN1}.

Using then \eqref{G_NN1} and the notations
$$
\hat \ff = (\dots, \hat f_{j-\frac{1}{2}}, \hat f_{j+\frac{1}{2}}, \dots),
\qquad
\mathbf{I} = (\dots, I_{j}, I_{j+1}, \dots)
$$
with componentwise operations, the conservative
scheme in \eqref{cons_form} can be given as a residual neural network with a skip
connection as shown in Figure \ref{res_fig}.

\begin{figure}[h]
    \centering
    \begin{tikzpicture}[
        node distance=2cm and 1.5cm,
        every node/.style={draw, minimum height=1cm, minimum width=1.5cm, align=center},
        skip/.style={-{Latex[length=2mm]}, thick},
        layer/.style={rectangle, draw=blue!60, fill=blue!20},
        input/.style={rectangle, draw=black!60, fill=gray!20},
        output/.style={rectangle, draw=green!60, fill=green!20}
      ]

      \node[input] (input) {Input\\ $\uu^n$};
      \node[layer, right=of input] (layer1) {Layers\\ in \eqref{G_NN1}};
      \node[layer, right=of layer1] (layer2) {Fluxes \\ $\hat\ff$};
      \node[output, right=of layer2] (add) {Output \\ $\uu^{n+1} = \uu^n - \frac{\delta}{\mathbf{I}} \hat\ff$};

      \draw[->, thick] (input) -- (layer1);
      \draw[->, thick] (layer1) -- (layer2);
      \draw[->, thick] (layer2) -- (add);

      \draw[skip] (input.north) .. controls +(up:1.2cm) and +(up:1.2cm) .. (add.north);

    \end{tikzpicture}
    \caption{The residual block for computing a single time step using a skip connection.}\label{res_fig}
\end{figure}
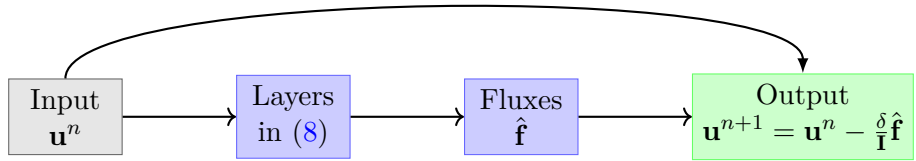
Finally, a consecutive application of the operations in Figure \ref{res_fig} gives the full conservative numerical method. This final network form is shown in Figure \ref{res_fig_full}.

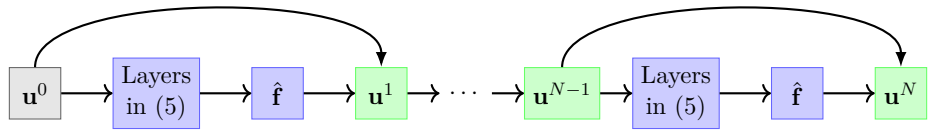
\begin{figure}[h]
    \centering
    \begin{tikzpicture}[
        scale=0.85, transform shape,
        node distance=0.6cm and 0.8cm,
        every node/.style={draw, minimum height=0.8cm, minimum width=0.8cm, align=center},
        skip/.style={-{Latex[length=2mm]}, thick},
        layer/.style={rectangle, draw=blue!60, fill=blue!20},
        input/.style={rectangle, draw=black!60, fill=gray!20},
        output/.style={rectangle, draw=green!60, fill=green!20}
      ]

      \node[input] (input1) {$\uu^0$};
      \node[layer, right=of input1] (layer1a) {Layers\\in (5)};
      \node[layer, right=of layer1a] (layer2a) {
      $\hat\ff$};
      \node[output, right=of layer2a] (out1) {$\uu^1$};

      \draw[->, thick] (input1) -- (layer1a);
      \draw[->, thick] (layer1a) -- (layer2a);
      \draw[->, thick] (layer2a) -- (out1);
      \draw[skip] (input1.north) .. controls +(up:1.2cm) and +(up:1.2cm) .. (out1.north);

      \node[draw=none, right=.5cm of out1] (dots) {$\cdots$};

    \node[output, right=.5cm of dots] (layer1b) {$\uu^{N-1}$};
    \node[layer, right=.5cm of layer1b] (layer1bb) {Layers\\ in (5)};
      \node[layer, right=of layer1bb] (layer2b) {$\hat\ff$};
      \node[output, right=of layer2b] (out2) {$\uu^N$};

      \draw[->, thick] (out1.east) -- (dots.west);
      \draw[->, thick] (dots.east) -- (layer1b);
      \draw[->, thick] (layer1b) -- (layer1bb);
      \draw[->, thick] (layer1bb) -- (layer2b);
      \draw[->, thick] (layer2b) -- (out2);
      \draw[skip] (layer1b.north) .. controls +(up:1.2cm) and +(up:1.2cm) .. (out2.north);

    \end{tikzpicture}
    \caption{The neural network for the fully discretized numerical solution corresponding to \eqref{cons_form}.}
    \label{res_fig_full}
\end{figure}

\subsubsection{Construction of reconstruction-based fluxes}

{Using the continuity of $\Phi$ in Section} \ref{sec_111},
{we have $\Phi(x_0)=\Phi(0) = 0$ and also, for an arbitrary $k\in \{1,2,\dots, K\}$ },
\begin{equation*}
    \Phi(x_k) = 
    \sum_{j=1}^k
    m_{j-1}\cdot (x_j-x_{j-1}).
\end{equation*}
Accordingly, for $k\in \{0,1,\dots, K-1\}$, we also define the sets
\begin{equation}\label{def_T_k}
T_k = \left\{(a,b): a(1-x_{k+1}) - bx_{k+1}
\le 0 \le
a(1-x_{k}) - bx_{k} \right\}\subset[0,\infty)^2,
\end{equation}
which are shown in Figure \ref{fig_T_k}.
Similarly, we need the following cones in $(-\infty,0]^2$:
$$
T_{-k} = \left\{(a,b):  bx_{k+1} - a(1-x_{k+1})
\le 0 \le
 bx_{k} - a(1-x_{k})  \right\}.
$$

\begin{figure}[ht]
\centering
\begin{tikzpicture}[scale=5]

  \draw[->] (0,0) -- (1.05,0) node[right] {$a$};
  \draw[->] (0,0) -- (0,1.05) node[above] {$b$};

  \fill[blue!20]
    (0,1) -- (0,0) -- ({1/3},1);

  \fill[green!30]
    (0,0) -- (1,1) -- (1,1/3) -- (0,0) -- cycle;

  \fill[orange!30]
    (0,0) -- (1,1) -- (1,1) -- ({1}, {1/3}) -- cycle;

  \fill[red!30]
    (0,0) -- (1,0) -- (1,{1/3}) -- cycle;

  \draw[dashed, thick] (0,0) -- (1,1) node[anchor=north west] {$b = a$};
  \draw[dashed, thick] (0,0) -- (1,{1/3}) node[anchor=north west] {$b = \frac{a}{3}$};
  \draw[dashed, thick] (0,0) -- ({1/3},1) node[anchor=south east] {$b = 3a$};

  \node at (0.1,0.9) {$T_0$};
  \node at (0.4,0.65) {$T_1$};
  \node at (0.7,0.4) {$T_2$};
  \node at (0.8,0.1) {$T_3$};

\end{tikzpicture}
\caption{Partition of the first quadrant into regions \(T_0\) to \(T_3\) based on linear inequalities for \(a\) and \(b\) in
\eqref{def_T_k}.}
\label{fig_T_k}
\end{figure}
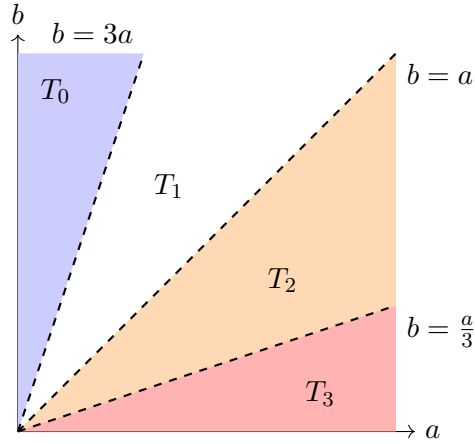

Note that relating the geometric representations in Figure \ref{pw_lin} and \ref{fig_T_k}, we also have
\begin{equation}\label{r1}
r\in [x_k, x_{k+1}]\Leftrightarrow (a,b)\in T_k
\quad\textrm{for}\;(a,b)\subset [0,\infty)^2
\end{equation}
and similarly,
\begin{equation}\label{r2}
r\in [x_k, x_{k+1}]\Leftrightarrow (a,b)\in T_{-k}
\quad\textrm{for}\;(a,b)\subset (-\infty, 0]^2.
\end{equation}

Using \eqref{r1} and \eqref{r2} and inserting \eqref{phi_def} into \eqref{slope_basic} gives that for all
$(a,b)\in T_k\cup T_{-k}$, we can rewrite the slope $s$  as
\begin{equation}\label{s_final}
\begin{aligned}
    s(a,b)&= (a+b) (m_k\cdot \left(\frac{a}{a+b} - x_k\right) + \Phi(x_k)) \\
   &= a\cdot m_k - a\cdot m_kx_k + a\cdot\Phi(x_k) - b\cdot m_kx_k + b\cdot\Phi(x_k) \\
    &=
    a\cdot(\Phi(x_k) + m_k - m_k x_k) + b\cdot(\Phi(x_k)-m_kx_k).
\end{aligned}
\end{equation}
Our neural network is based on this representation with the inputs $a$ and $b$.
In case of $a,b\ge 0$, {we seek the slope function $s$ in the form}
\begin{equation*}
    s(a,b) =
    \sum_{k=0}^n t_k\cdot \textrm{ReLu}((1-x_k)\cdot \textrm{ReLu}\:(a) - x_k\cdot \textrm{ReLu}\: (b)).
\end{equation*}
Regarding this, we have the following statement.
\begin{theorem}
    Using the parameters $t_k = m_k-m_{k-1}$,
the function
$
s=s_1-s_3:\er^2\to\er
$
with
$$
\begin{aligned}
s_1(a,b) &=
m_0\cdot \textrm{ReLu}\:(a)\\
&+
    \sum_{k=1}^{n-1} (m_k-m_{k-1})\cdot \textrm{ReLu}\:((1-x_k)\cdot\textrm{ReLu}\: (a)  - x_k\cdot\textrm{ReLu}\: (b))\\
\end{aligned}
$$
 and
$$
\begin{aligned}
s_3(a, b) &=  - s_1(-a, -b) =
-m_0\cdot \textrm{ReLu}\:(-a) \\ &-
    \sum_{k=1}^{n-1} (m_k-m_{k-1})\cdot
    \textrm{ReLu}\:((1-x_k)\cdot \textrm{ReLu}\:(-a) -x_k\cdot\textrm{ReLu}\:(-b) )
\end{aligned}
$$
gives the one defined in \eqref{s_final}.
\end{theorem}
\emph{Proof:} {Applying case analysis,}
we first investigate $s_1$ in the four quadrants of $\er^2$.
\begin{itemize}
    \item
 { In case of $a,b>0,$ we have that $(a,b)\in T_j$ for some $j\in\{0,1,\dots,n\}$,
such that}
$$
a(1-x_k) - bx_k\ge 0, \quad k = 0, 1, 2,\dots, j,
$$
and
$$
a(1-x_k) - bx_k\le 0, \quad k = j+1,j+2,\dots, n.
$$
In this way,
\begin{equation}\label{s11}
\begin{aligned}
s_1(a,b) &=
    m_0\cdot a +
    \sum_{k=1}^{j} (m_k-m_{k-1})\cdot (a(1-x_k) - bx_k)\\
    &=
     a\cdot\left(m_0 x_1 + m_1 (x_2-x_1) + \dots + m_{j-1} (x_j-x_{j-1}) + m_j(1-x_j)\right)\\
&+
 b\cdot \left(m_0 x_1 + m_1 (x_2-x_1) + \dots + m_{j-1} (x_{j}-x_{j-1}) - m_jx_j\right)\\
    &=a\cdot(\Phi(x_j) + m_j(1-x_j))
    + b\cdot \left(\Phi(x_j) - m_jx_j\right),
\end{aligned}
\end{equation}
{which really coincides with \eqref{s_final}.}

  \item
In case of { $a<0$,  we obtain
\begin{equation}\label{a_neg}
 s_1(a,b) =\begin{cases}
      \displaystyle{\sum_{k=1}^{j} (m_k-m_{k-1})\cdot\relu((-x_k)\cdot\relu( b)) = 0 \quad \textrm{for}\; b> 0},\\
     0 \quad \textrm{for}\; b< 0.
  \end{cases}
\end{equation}
}
\item
In case of $a>0$ and $b<0$, the relation $\relu (b) = 0$ holds, such that based on \eqref{s11} with $j=n-1$, we get
\begin{equation}\label{a_pos}
\begin{aligned}
s_1(a,b) &=
m_0\cdot a +
    \sum_{k=1}^{n-1} (m_k-m_{k-1})\cdot  a\cdot (1-x_k)  \\
    &=
    { a\cdot (\Phi(x_{n-1}) + m_{n-1}\cdot(1-x_{n-1}) ) = a\cdot \Phi(1)=0.}
\end{aligned}
\end{equation}
\end{itemize}
Second, we investigate $s_3$ in the four quadrants of $\er^2$.
\begin{itemize}
    \item
In case of { $a>0$, the identity in \eqref{a_neg}, while for $a<0, b>0$, \eqref{a_pos} implies} that
$
s_3(a,b) = - s_1(-a,-b) = 0 $.
\item
Finally, in case of of $a<0$ and $b<0$,  \eqref{s_final} gives
$$
s_3(a,b) = - s_1(-a,-b) = -s(-a,-b) = -s(a,b).
$$
\end{itemize}
Summarized, in each quadrant, we indeed have $s_1(a,b) - s_3(a,b) = s(a,b)$, as stated in the theorem.\quad$\square$\medskip

The neural network corresponding to the present section is shown in Figure \ref{fig:slope_block}. Here $n=2$ and the linear transformation
$\displaystyle{\mathrel{\mathop{\textrm{lin}}_{1,-1}}}$ corresponds to
$a_j = u^n_j - u^n_{j-1}$ and $b_j = u^n_{j+1} - u^n_{j}$, while
$\displaystyle{\mathrel{\mathop{\textrm{lin}}_{\textrm{full}}}}$ to the final summation in the definition of $s_1$ and $s_3$, respectively.

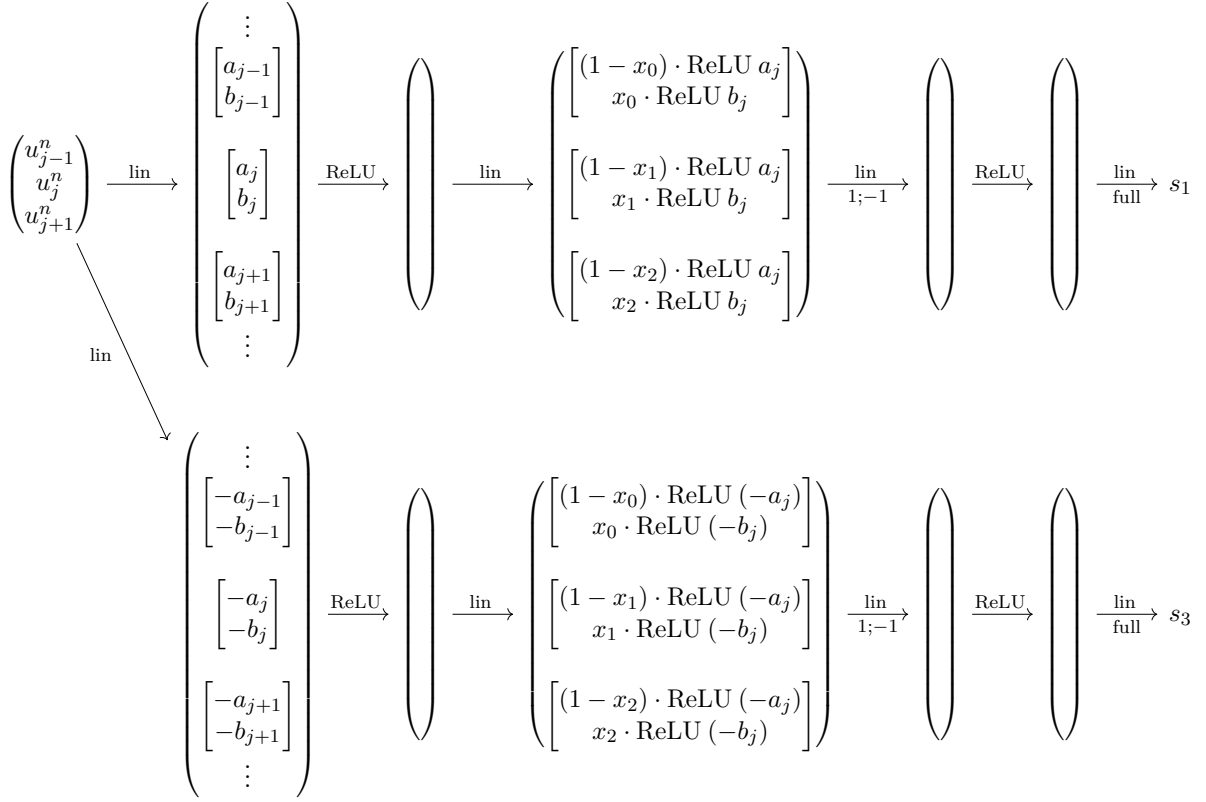
\begin{figure}[htbp]
  \centering
  \begin{adjustbox}{width=\textwidth}
\begin{tikzcd}
\begin{pmatrix}
        u_{j-1}^n\\u_j^n\\u_{j+1}^n
    \end{pmatrix}
    \arrow[r, "\lin"]\arrow[dr, "\lin"']
    & \begin{pmatrix}
        \vdots\\
        \begin{bmatrix}
        a_{j-1}\\ b_{j-1}
    \end{bmatrix}\\
    {}
    \\ \begin{bmatrix}
        a_j\\ b_j
    \end{bmatrix} \\{}\\
     \begin{bmatrix}
        a_{j+1}\\ b_{j+1}
    \end{bmatrix}\\
    \vdots
    \end{pmatrix}
   \arrow[r, "\relu"]
   &
    \begin{pmatrix}
        {}\\{}\\{}\\{}\\{}\\{}\\{}\\{}
    \end{pmatrix}
    \arrow[r,"\lin"]
    &
    \begin{pmatrix}
        \begin{bmatrix}
        (1-x_0)\cdot \relu a_j\\x_0\cdot \relu b_j
    \end{bmatrix}
    \\{}\\
    \begin{bmatrix}
        (1-x_1)\cdot\relu a_j\\x_1\cdot \relu b_j
    \end{bmatrix}
    \\{}\\
    \begin{bmatrix}
        (1-x_2)\cdot\relu a_j\\x_2\cdot\relu b_j
    \end{bmatrix}
    \end{pmatrix}
    \arrow[r,"\lin","1; -1"'] & \begin{pmatrix}
        {}\\{}\\{}\\{}\\{}\\{}\\{}\\{}
    \end{pmatrix}
    \arrow[r,"\relu"] & \begin{pmatrix}
        {}\\{}\\{}\\{}\\{}\\{}\\{}\\{}
    \end{pmatrix}
    \arrow[r,"\lin", "\textrm{full}"'] &
       s_1
    \\
                        &
                \begin{pmatrix}
        \vdots\\
        \begin{bmatrix}
        -a_{j-1}\\ -b_{j-1}
    \end{bmatrix}\\
    {}
    \\ \begin{bmatrix}
        -a_j\\ -b_j
    \end{bmatrix} \\{}\\
     \begin{bmatrix}
        -a_{j+1}\\ -b_{j+1}
    \end{bmatrix}\\
    \vdots
    \end{pmatrix}
    \arrow[r, "\relu"]
          &
    \begin{pmatrix}
        {}\\{}\\{}\\{}\\{}\\{}\\{}\\{}
    \end{pmatrix}
    \arrow[r,"\lin"]
    &
    \begin{pmatrix}
        \begin{bmatrix}
        (1-x_0)\cdot \relu (-a_j)\\x_0\cdot \relu (-b_j)
    \end{bmatrix}
    \\{}\\
    \begin{bmatrix}
        (1-x_1)\cdot\relu (-a_j)\\x_1\cdot \relu (-b_j)
    \end{bmatrix}
    \\{}\\
    \begin{bmatrix}
        (1-x_2)\cdot\relu (-a_j)\\x_2\cdot\relu (-b_j)
    \end{bmatrix}
    \end{pmatrix}
    \arrow[r,"\lin","1; -1"'] & \begin{pmatrix}
        {}\\{}\\{}\\{}\\{}\\{}\\{}\\{}
    \end{pmatrix}
    \arrow[r,"\relu"] & \begin{pmatrix}
        {}\\{}\\{}\\{}\\{}\\{}\\{}\\{}
    \end{pmatrix}
    \arrow[r,"\lin", "\textrm{full}"'] & s_3
\end{tikzcd}
\end{adjustbox}

\caption{A diagram showing the subnetwork, which generates the reconstructed slopes with $s=s_1+s_3$.}
  \label{fig:slope_block}
\end{figure}

\subsection{The complete neural network and loss function}
Using the reconstructed  $s=s_1+s_3$, we complete the network in \eqref{G_NN1} as follows.

\begin{equation*}
 \begin{pmatrix}
     u_{j-1}^n\\u_j^n\\u_{j+1}^n
    \end{pmatrix}
    \rightarrow
   \begin{pmatrix}
     s_1\\s_3
    \end{pmatrix}
 \rightarrow
    \begin{pmatrix}
     u_j^n - \frac{hs}{2}\\
     u_j^n + \frac{hs}{2}
    \end{pmatrix}
    \xlongrightarrow{A\cdot}
    \begin{pmatrix}
     {\;}\\{\;}
    \end{pmatrix}
     \xlongrightarrow{ \textrm{ReLu}}
    \begin{pmatrix}
    {\;}\\{\;}
    \end{pmatrix}
     \xlongrightarrow{f}
    \begin{pmatrix}
    {\;}\\{\;}
    \end{pmatrix}
     \xlongrightarrow{\max}
     \hat f_{j+\frac{1}{2}},
    \end{equation*}
  where in the first step, we have used the network in Figure  \ref{fig:slope_block}.
{ The successful training of a neural network requires an appropriate loss function.} Inspired by the theory of the conventional numerical methods, in either case, we used the following one.
\begin{equation}\label{loss_def}
\begin{aligned}
&L:\er^{N}\to\er,\\
&L(u_{\textrm{pred},t}) = \|u_{\textrm{pred},t} - u(t,\cdot)\|_2^2
+ w\cdot
(\|u_{\textrm{pred},t}\|_{\TV} -
\|u(t, \cdot)\|_{\TV})^2,
\end{aligned}
\end{equation}
where $\|\cdot\|_{\TV}:\er^N\to\er$ denotes total variation seminorm with
$$
\|\mathbf{w}\|_{\TV} =
\sum_{j=1}^{N-1} |w_{j+1}-w_j|,
$$
and $u(t,\cdot)$ denotes the analytic solution at the gridpoints as the
training data. Finally, $u_{\textrm{pred},t}$  is for the output of the
neural network-based numerical solution.\\
In this way, we can enforce the TVD property of our scheme. We also know that TVD schemes are nonlinear and avoid spurious oscillations, but they are necessarily only first-order accurate at extrema. To balance between these properties, we use
the weight $w$ in \eqref{loss_def}.

\subsection{Numerical experiments}
We perform numerical experiments to train the network parameters and
compute the corresponding numerical solution of the one-dimensional scalar conservation law
$$
\begin{cases}
\partial_t u(t,x) + f(u)(t,x) = 0\quad (t,x)\in (0,{ t_{\textrm{end}}})\times (-1,1)\\
u(0,x) = u_0(x)\quad x\in (-1,1),
\end{cases}
$$
which should be equipped with a boundary condition depending on the initial data.

\subsubsection{Model problems and discretization parameters}\label{disc_par_sect}

In the first series of experiments,
we have investigated the one-dimensional Burgers problem with $f(u) = \frac{u^2}{2}$ so that the conditions in (A1)-(A3) are satisfied.

In the second series of numerical experiments, we have used the flux function $f(u) = u - u^2$ in \eqref{cons_law}. This corresponds to the Lighthill--Whitham--Richards (LWR) traffic flow model \cite{lighthill55}, \cite{richards56}.
Since this function $f$ is strictly concave, we can again apply \eqref{G_scheme}. Accordingly, following the lines in Proposition 1, one can show that in this case, \eqref{G_scheme} can be given as
\begin{equation} \label{LWR_flux}
\begin{aligned}
&\hat f_{j+\frac{1}{2}}
=
\min
\left\{ f
\left(-\relu\!\left(\frac{1}{2} - u^{L}_{j+\frac{1}{2}}\right) + \frac{1}{2}
\right),\right.\\
&\qquad\qquad\left.
f
\left(\relu\!\left( u^{R}_{j+\frac{1}{2}}-\frac{1}{2}\right) + \frac{1}{2}
\right)
\right\}.
\end{aligned}
\end{equation}
Here, we regard $\frac{1}{2}$ as a known parameter, which is the critical point of the flux function $f$.

{ For the results in Section \ref{learning_G_sect}, the spatial discretization of $[-1,1]$ consists of  uniformly distributed grid points of number $\textrm{nx} =129.$
In Section \ref{learning_G_sect}, we have performed 64 time steps with a size of
$\delta =\frac{0.25}{64}$, according to the final time $t_{\textrm{end}}=0.25$.

In Section \ref{learning_R_sect}, in most cases, the parameters nx, $\delta$ and $t_{\textrm{end}}$ were systematically varied to ensure a thorough evaluation.
Accordingly, these parameters are reported in the descriptions of the experiments.}

\subsubsection{Training dataset}

To construct the training dataset, we
considered model problems with the following initial data.
\begin{itemize}
\item Riemann-type:
\[
u_0(x) = \begin{cases}
u_l\quad x < x_0 \\
u_r\quad x > x_0
\end{cases}
\]
\item Piecewise linear:
\[
u_0(x) = \begin{cases}
u_l\quad x < x_1 \\
u_l + \frac{u_r - u_l}{x_2 - x_1}(x - x_1)\quad x_1 < x < x_2 \\
u_r\quad x > x_2
\end{cases}
\]
\item Trigonometric:
\[
u_0(x) = r_3 \sin(\pi x + r_2).
\]
\end{itemize}
Here the parameters were sampled uniformly in the following intervals
\begin{itemize}
    \item
    $u_l, u_r \in [-1, 1],\; x_0 \in [-0.25, 0.25]$
    \item
    $x_1 \in [-0.25, 0.25], \quad x_2 \in [x_1, x_1 + 0.5]$
    \item
    $r_2 \in [0, 2\pi], \quad r_3 \in [0, 1]$.
\end{itemize}

While in the first two cases, the analytic solution is known, for the trigonometric initial data, we have used an accurate approximation of this;
see \ref{secA1}.

In this case, the solution exhibits smooth behavior at the beginning of the
process, then a shock appears {delivering an important component of the training set.\\
Following the approach of \cite{chen24}, each training sample was initialized at a randomly selected time and evolved over a time interval that includes the onset of shock formation.}


The boundary conditions were always set to match the exact solution at the boundaries.

The training dataset comprised 1,200 problems, evenly distributed among the three categories of initial values described above. Of these, 1,000 instances were allocated for training and 200 for validation.

\subsubsection{Training procedure}\label{training_sect}

{ In all cases, the input of the network consists of the initial and boundary conditions, and the output is the solution at $t_{\textrm{end}}$.} The training was carried out in each case, using the Adam optimizer.

{For learning Godunov's scheme in Section \ref{learning_G_sect}, the network learned over 100 epochs with the batch size 25.}

{ For learning the reconstruction schemes in Section \ref{learning_R_sect}, training was terminated at the first epoch n when the relative change in the error over the preceding five epochs dropped below
$10^{-3}$:
$$
\frac{|\textrm{error}_n - \textrm{error}_{n-5}|}{\textrm{error}_{n-5}}
< 0.001.
$$
Here various epoch sizes were applied: a special series of experiments
are displayed in Table \ref{table_gr} summarizing the details in this
section.}

The corresponding Python {codes} can be run on a simple laptop, we used
one with an Intel Core i7-7500U CPU.
To speed up this procedure, we have also run the procedure on the supercomputer \emph{Komondor} using a GPU of type NVIDIA A100-SXM4-40GB.

\subsubsection{Learning Godunov's scheme}\label{learning_G_sect}
An interesting consequence of Proposition \ref{basic_prop} is that without any prior knowledge, we can learn Godunov's scheme. In concrete terms, if the entries of $A\in\er^{2\times 2}$ in \eqref{G_NN1} are considered unknown parameters, training the corresponding neural network can reproduce the choice in Proposition \ref{basic_prop}.
This network can possibly be deep by incorporating a number of time steps as shown in Figure \ref{res_fig_full}. To demonstrate the performance of the
learning, we compare the matrix $A\in\er^{2\times 2}$ in \eqref{G_NN1} with the one obtained in the learning process. For this, we use the Frobenius norm of their difference, see Figure \ref{learn_G_method}. Also, the training and validation loss are shown here, which exhibit the same decay.
\vspace{-20pt}
\begin{figure}[H]
\centering
\includegraphics[width=.49\linewidth]{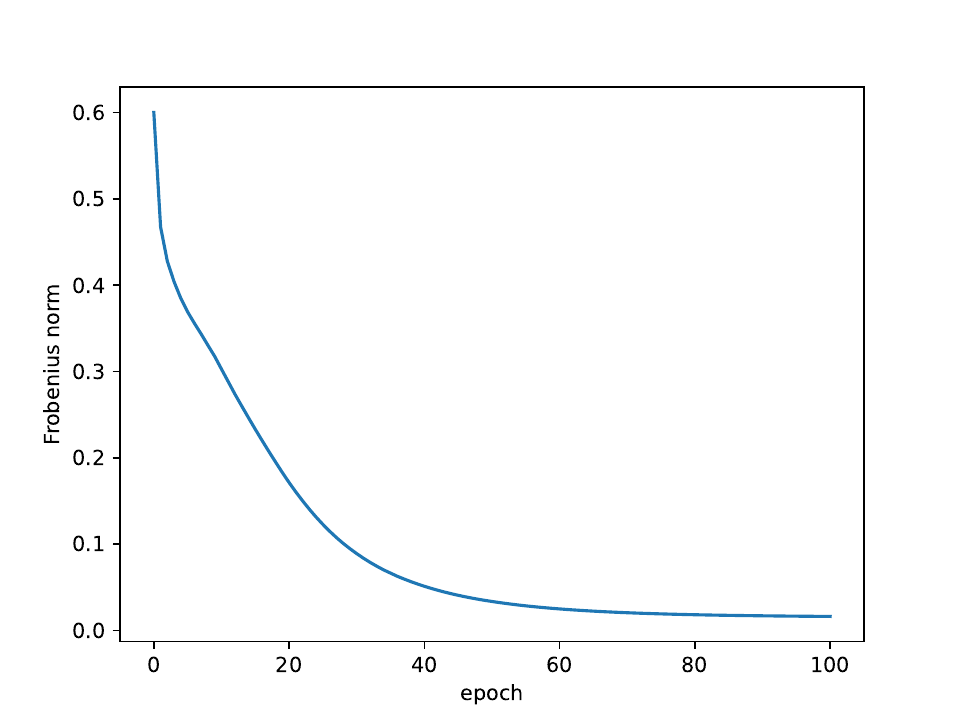}\hfill
\includegraphics[width=.49\linewidth]{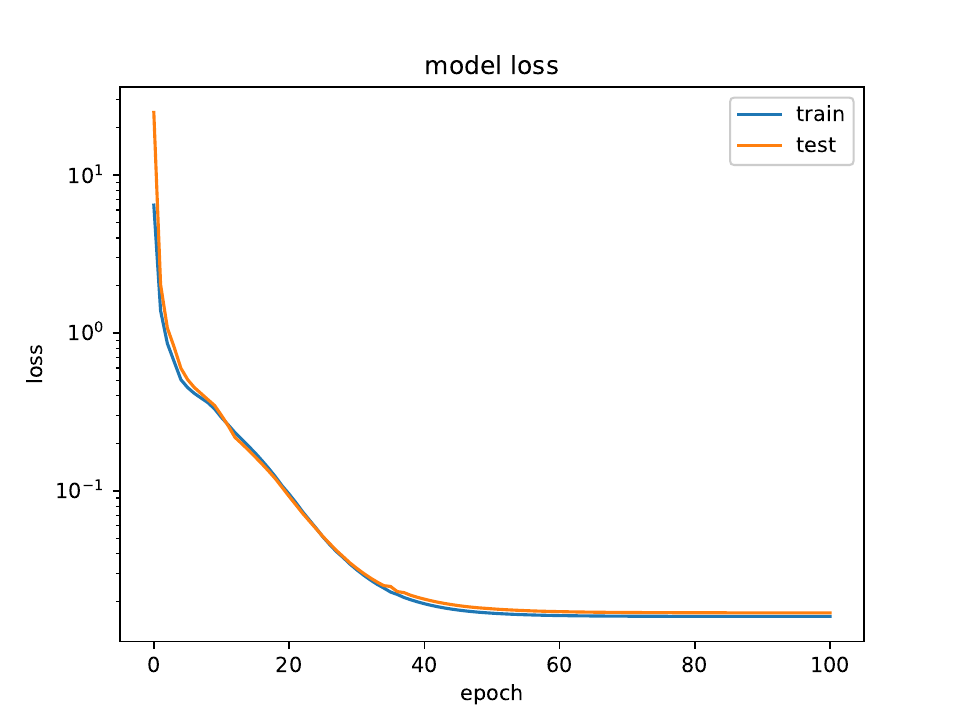}
\vspace{-5pt}
\caption{Learning the Godunov scheme. Left: the Frobenius norm of the difference of the matrix $A$ in  Proposition \ref{basic_prop} and the one in the training process. Right: training and validation losses over 100 epochs.} \label{learn_G_method}
\end{figure}

The process was initiated with the matrix $A_0=\begin{pmatrix}
    0.7&0.3\\
    -0.3&-0.7
\end{pmatrix}$. The full simulation results in the case of the trigonometric initial data with this initial choice and with the final choice are shown in Figure
\ref{res_with_A}. Here, we used the initial condition
$$
u_0(x) = 0.62\cdot \sin(\pi x + 2.27)
$$
at the initial time  $t_0=0.4375$. A shock arises at $t\approx 0.5$ and the final time  is $0.6875$.
\begin{figure}[H]
\centering
\includegraphics[width=.49\linewidth]{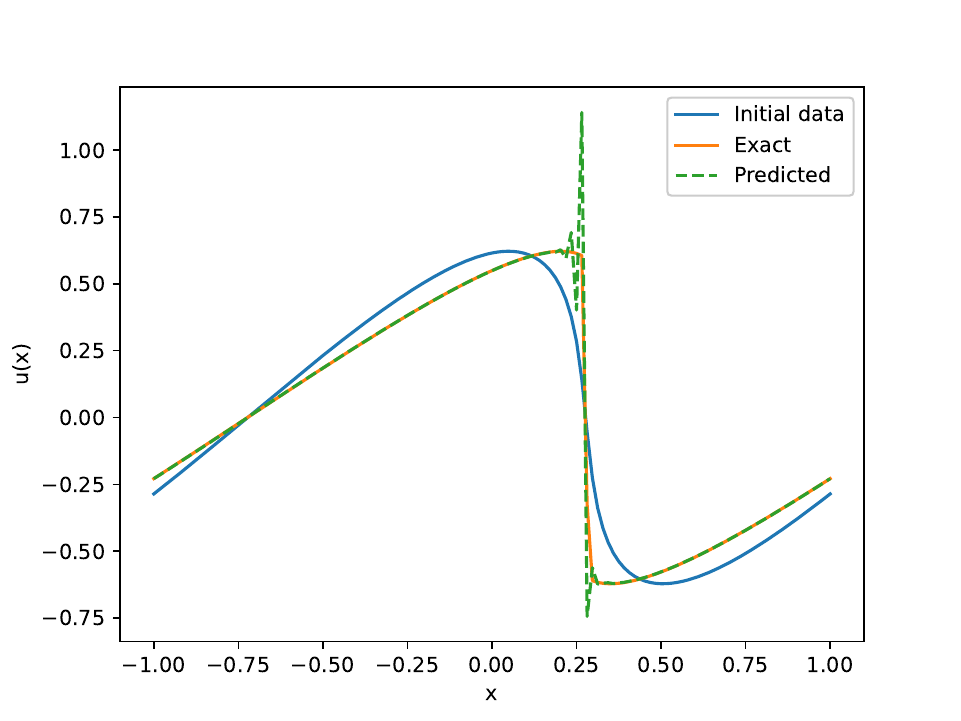}\hfill
\includegraphics[width=.49\linewidth]{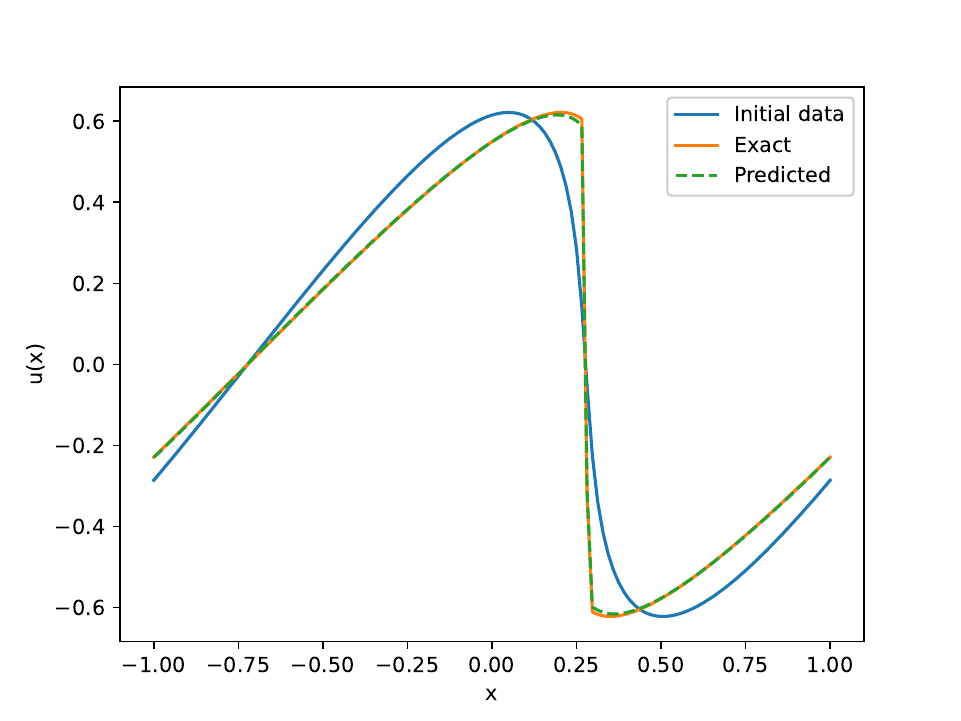}
\caption{Result of the full simulation corresponding to Figure \ref{res_fig} using the initial matrix $A_0$ and the final one after 100 epochs. The trigonometric test problem was used, with the simulation run for 0.1875 time units after the shock developed.} \label{res_with_A}
\end{figure}

According to \eqref{LWR_flux}, in the case of the LWR traffic model, we consider $\frac{1}{2}$ in the bias terms as a known parameter and we intend to learn only the linear transformation
$
\begin{pmatrix}
    u^L_{j+\egyk}\\u^R_{j+\egyk}
\end{pmatrix}
\xrightarrow{A\cdot}
\begin{pmatrix}
    -u^L_{j+\egyk}\\u^R_{j+\egyk}
\end{pmatrix}.
$
Here the multiplier $A$ has again 4 parameters, which is given in the second row of the table. Also, the network could  learn the bias terms and the other linear mapping
$$
\begin{pmatrix}
    \relu\!\left(\frac{1}{2} - u^{L}_{j+\frac{1}{2}}\right) \\
\relu\!\left( u^{R}_{j+\frac{1}{2}}-\frac{1}{2}\right)
\end{pmatrix}
\to
\begin{pmatrix}
-\relu\!\left(\frac{1}{2} - u^{L}_{j+\frac{1}{2}}\right) + \frac{1}{2}\\
\relu\!\left( u^{R}_{j+\frac{1}{2}}-\frac{1}{2}\right)+ \frac{1}{2}
\end{pmatrix}.
$$
Here, again, the entries of a $2\times 2$ matrix and a bias term would give 6 parameters such that, altogether, we would have  $6+6 = 12$ of them.


\begin{table}[H]
\caption{Summary of the main characteristics of neural networks and performance of the learning process for
getting optimal conservative
schemes in case of the
Burgers and the LWR equation.
}
\label{table_gr}
\centering
\begin{tabular}
{@{\extracolsep\fill}lccccccc}
\toprule%
& & & \multicolumn{3}{@{}c@{}} {Number/size of} & \\
\cmidrule{4-7}%
Problem & Method & Computer & \makecell{Input$\&$\\output} & Layers & \makecell{Para-\\meters} & Epochs & Time \\
\midrule
Burgers & Godunov & Laptop & 128 & 64 & 4 & 100 & 355 s\\
Burgers & Godunov & HPC & 128 & 64 & 4 & 100 & 235 s\\
LWR & Godunov & Laptop & 128 & 64 & 4 & 100 & 477 s\\
LWR & Godunov & HPC & 128 & 64 & 4 & 100 & 315 s\\
Burgers & reconstruction & Laptop & 128 & 64 & 5 & 20 & 240 s\\
Burgers & reconstruction & HPC & 128 & 64 & 5 & 20 & 142 s\\
LWR & reconstruction & Laptop & 128 & 64 & 5 & 20 & 246 s\\
LWR & reconstruction & HPC & 128 & 64 & 5 & 20 & 161 s\\
\midrule
\end{tabular}
\end{table}

\FloatBarrier
\subsubsection{{Learning reconstruction schemes}}\label{learning_R_sect}

For learning the slopes in \eqref{final_rec_step},
 we have fixed the breakpoints $(x_0,x_1,\dots,x_6) = [0,\frac{1}{4}, \frac{1}{3},\frac{1}{2},\frac{2}{3},\frac{3}{4},1]$, see \eqref{phi_def} and Figure \ref{pw_lin}. In this way, as unknown parameters, we applied
 $\Phi\left(\frac{1}{4}\right), \Phi\left(\frac{1}{3}\right),
 \Phi\left(\frac{1}{2}\right),
 \Phi\left(\frac{2}{3}\right),
 \Phi\left(\frac{3}{4}\right)$,
 which give the slopes
  $m_0, m_1, m_2, m_3$ and $m_4$.

The performance of the learning process regarding the reconstruction scheme is displayed in Figure \ref{fig_limiters_all}. Initiated from the classical superbee and the minimax limiters, the symmetry is preserved and tends to similar ones. In both cases, we used the parameters nx = 129 and $\delta =\frac {0.25}{64}$.

\begin{figure}[H]
\includegraphics[width=.5\linewidth]{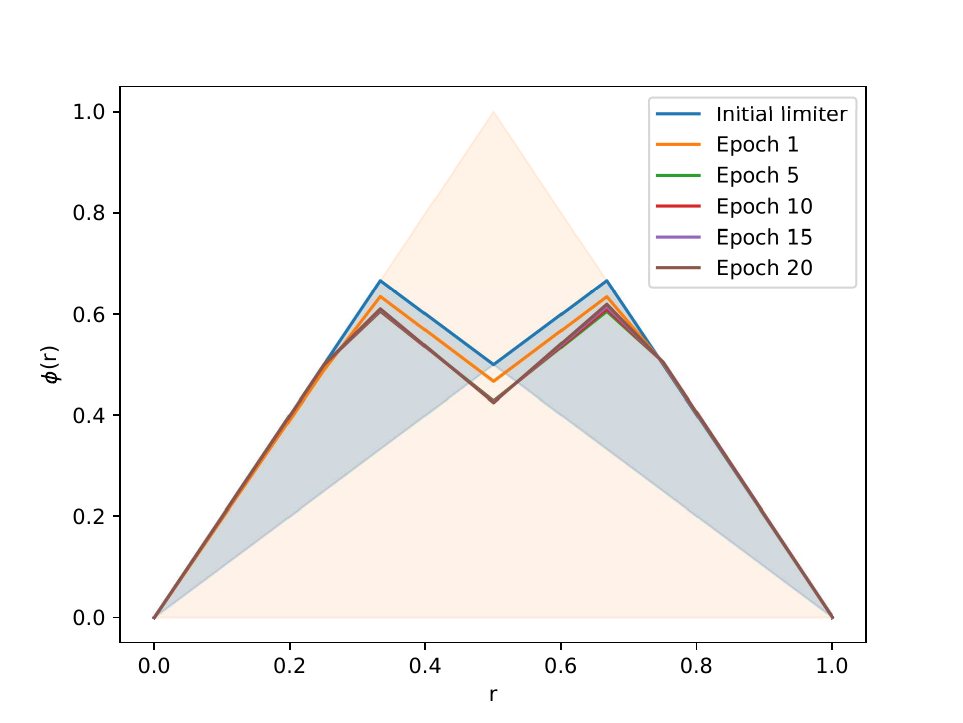}\hfill
\includegraphics[width=.5\linewidth]{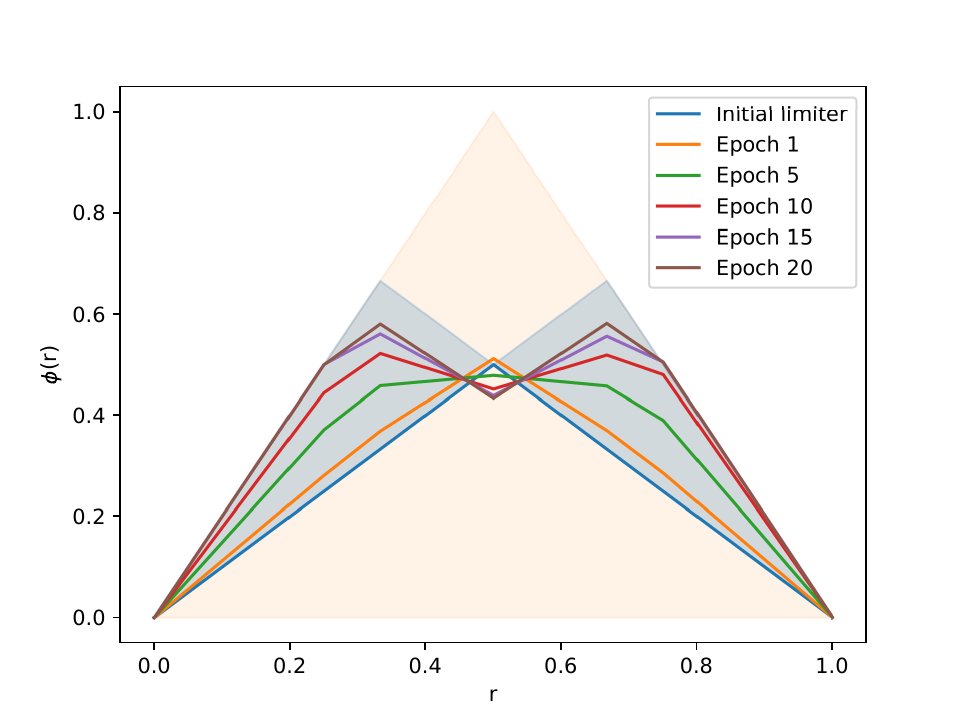}
\caption{Variation of slope limiters in the training procedure using two different initial parameters over 20 epochs. Left: starting from the superbee limiter, right: starting from the minmax limiter. In both cases, we used the parameters nx = 129 and $\delta =\frac {0.25}{64}$.}
\label{fig_limiters_all}
\end{figure}

Both for the Burgers equation and the LWR problem, we used the same reconstruction block shown in Figure \ref{fig:slope_block}. At the same time, the Godunov step was different for the two problems. As the parameters in this step have already been learned, in both cases, we have
only five learnable parameters. The setup of the network is shown again in Table \ref{table_gr}. Here, for comparison purposes, in each case, we again used the fixed parameter values nx = 129 and $\delta =\frac {0.25}{64}$ in all cases.

\paragraph{Optimality and robustness of the learned limiters}

We first investigated, on the inviscid Burgers equation, how the optimal reconstruction limiter depends on both the discretization and on the class of model problems used during training. The goal of this experiment was not only to obtain a single best-performing limiter, but also to understand whether the learned reconstruction is robust with respect to changes in mesh resolution, CFL number, initial data and final time. The following parameters were varied:
\begin{itemize}
    \item the number of spatial grid points, $\text{nx}\in\{33,65,129,257\}$;
    \item the CFL number, $\mathrm{CFL}\in\{1/2,1/4,1/8\}$, where
    $$
    \mathrm{CFL}=\frac{\delta}{h};
    $$
    \item the initial condition class, $u_0\in\{\text{Riemann},\text{linear},\text{trigonometric}\}$;
    \item the length of the time interval, $\Delta t\in\{1/4,1/2,1\}$.
\end{itemize}
The first two parameters determine the numerical discretization, while the latter two characterize the underlying test problem. We first studied the influence of the discretization parameters on the limiter obtained at the end of training. In this experiment, all three initial-condition classes and all three final times were included in the training set, giving $3\times 3=9$ different problem classes. These classes were placed in separate batches; hence each epoch consisted of 9 batches, each with batch size 100.

In the first discretization test, the spatial mesh was fixed and the CFL number was varied. The results are shown in Fig.~\ref{fig:fix_nx}. In this case, the spatial interval was divided into 256 cells. Smaller CFL numbers require more time steps to reach the same final time. As expected, sharper reconstructions are beneficial when the numerical error accumulated over many time steps has to be compensated.

In the second discretization test, the CFL number was fixed and the spatial resolution was varied. The corresponding results for $\mathrm{CFL}=0.25$ are displayed in Fig.~\ref{fig:fix_cfl}. The figure indicates that, as the mesh is refined, the optimal limiter becomes less aggressive. This behavior is consistent with the fact that, on finer grids, the truncation error is already smaller and the network does not need to compensate for coarse-grid diffusion by selecting an overly sharp reconstruction.

\begin{figure}[h]
    \centering
    \begin{subfigure}[b]{0.49\textwidth}
    \centering
        \includegraphics[width=\textwidth]{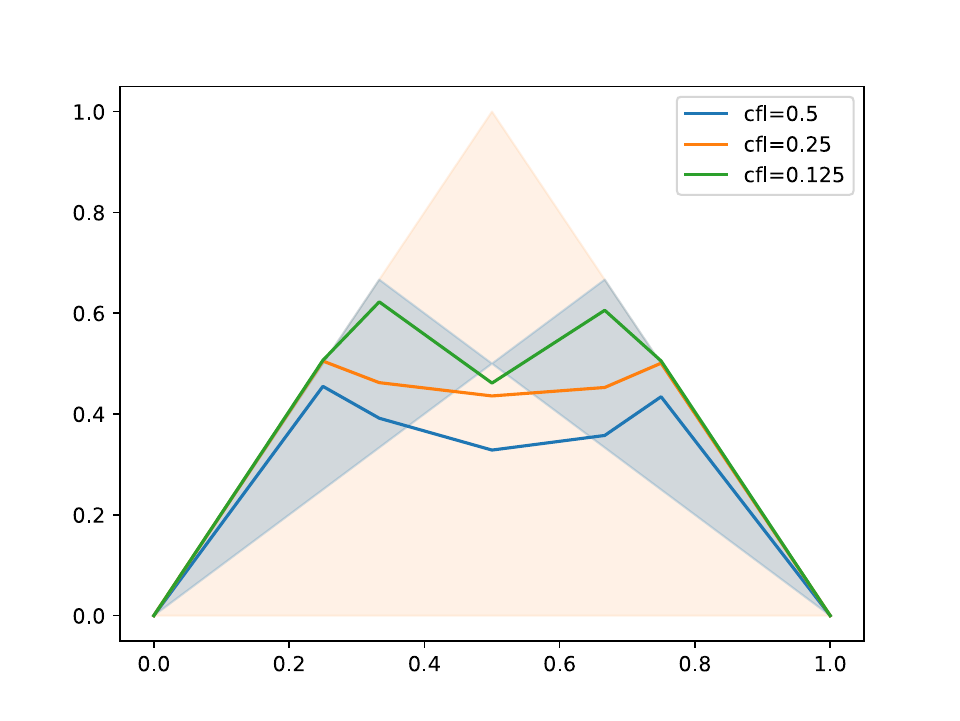}
        \caption{Fixed spatial discretization}\label{fig:fix_nx}
    \end{subfigure}
    \hfill
    \begin{subfigure}[b]{0.49\textwidth}
    \centering
        \includegraphics[width=\textwidth]{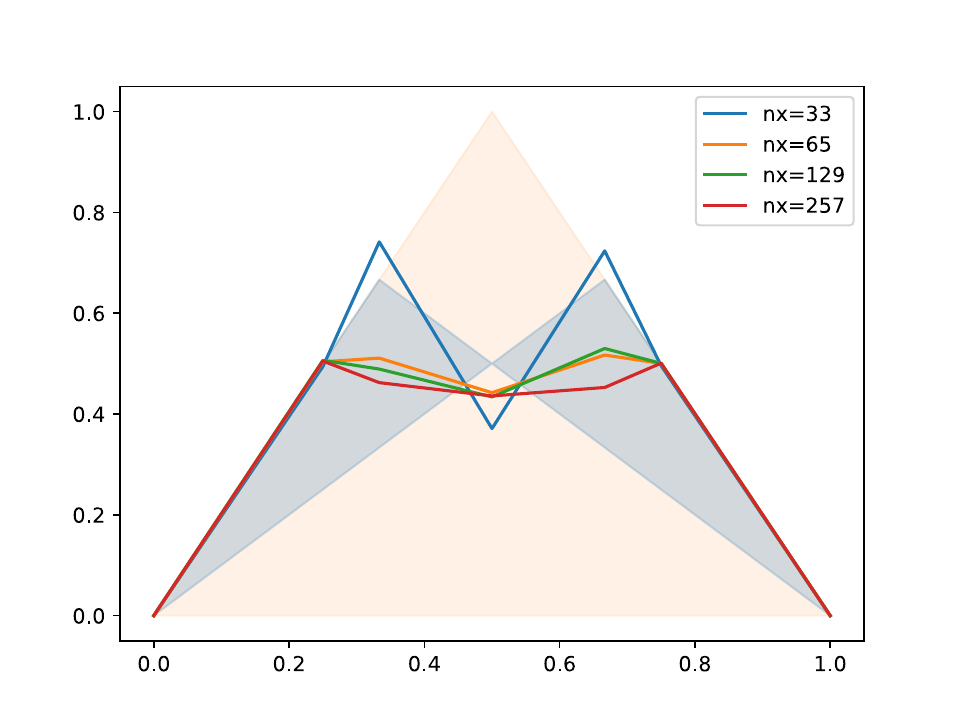}
        \caption{Fixed CFL number}\label{fig:fix_cfl}
    \end{subfigure}
    \caption{Learned slope limiters obtained with fixed spatial discretization nx = 257 (left) and with fixed CFL number 0.25 (right).}
\end{figure}

Next, we examined which limiter is preferred for fixed classes of model problems. In this part of the experiment, all CFL numbers were included in the training data, but only the finer spatial discretizations with $nx=129$ and $nx=257$ were used. This choice was motivated by the previous experiment: on very coarse meshes, the learned limiter tends to become extremely aggressive in order to reduce the dominant discretization error, which may obscure the dependence on the problem class itself. Thus, in this setting, the training set contained $3\times 2=6$ different discretizations. These were placed into separate batches, so that each epoch consisted of 6 batches of size 100.

First, the initial-condition class was fixed and the final time was varied. Figure~\ref{fig:fix_type} shows the results for trigonometric initial data. The outcome is counterintuitive at first sight: longer time intervals lead to more aggressive learned limiters. One might expect the opposite, since aggressive limiters can amplify non-physical oscillations over many time steps. However, this behaviour can be explained by the structure of the solution. Starting from smooth trigonometric data, the solution develops a shock after a predictable time and subsequently becomes closer to a piecewise smooth, almost piecewise linear profile separated by discontinuities. Such profiles can be approximated efficiently by sharper reconstructions, provided that the limiter remains sufficiently stable near shocks.

Finally, the final time was fixed and the class of initial data was varied. The results in Fig.~\ref{fig:fix_dt} show that, for trigonometric initial data, a more aggressive limiter may be advantageous compared to the Riemann and linear initial data.

The results confirm that the optimal reconstruction is not universal: it depends not only on the mesh and time-step size, but also on the qualitative structure of the solution over the time interval considered.

\begin{figure}[h]
    \centering
    \begin{subfigure}[b]{0.49\textwidth}
    \centering
        \includegraphics[width=\textwidth]{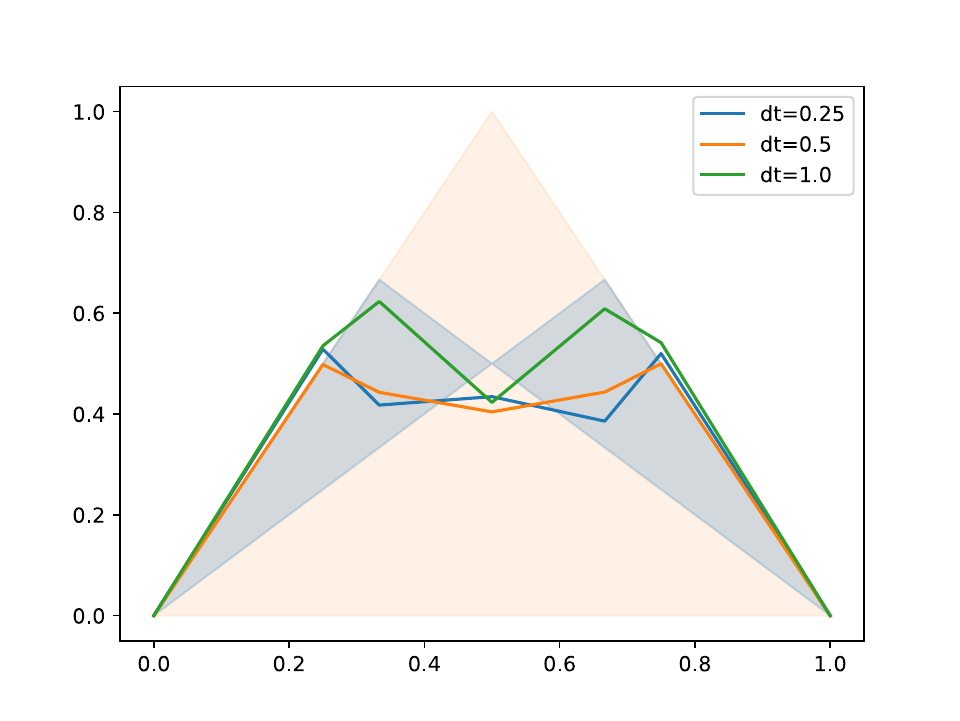}
        \caption{Fixed initial-condition class}\label{fig:fix_type}
    \end{subfigure}
    \hfill
    \begin{subfigure}[b]{0.49\textwidth}
    \centering
        \includegraphics[width=\textwidth]{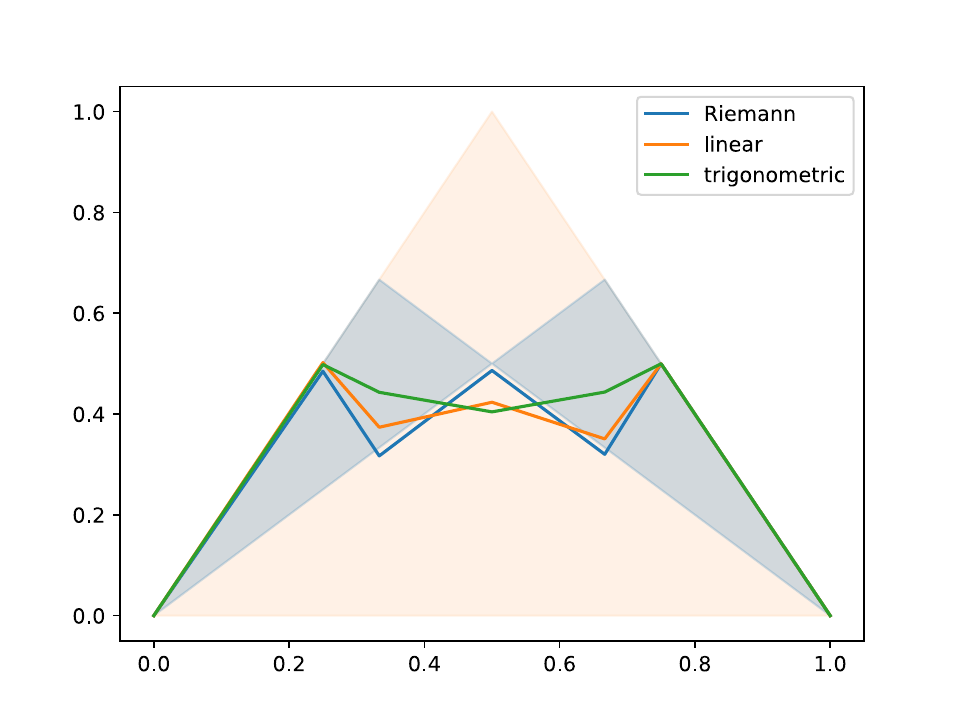}
        \caption{Fixed final time}\label{fig:fix_dt}
    \end{subfigure}
    \caption{Learned slope limiters obtained for trigonometric initial conditions (left) and fixed final time $dt =0.5$ (right).}
\end{figure}

\paragraph{{Comparison with standard limiters}}
The previous experiments show that there is no single limiter that is optimal for every discretization and for every class of initial data. Consequently, testing a network trained only on one type of problem, for example, Riemann problems, on a qualitatively different class such as trigonometric initial data would not give a fair assessment of the method. In the following experiments, we therefore evaluate the learned limiter on problem classes that were included during training.

\begin{table}[ht!]
\centering
\caption{Discrete $L_2$ errors for different limiters and CFL numbers. The reported values are scaled by $10^{-3}$.}
\label{tab:cfl_limiters}
\begin{tabular}{cc S[table-format=2.2] S[table-format=2.2] S[table-format=2.2] S[table-format=2.2]}
\toprule
\multirow{2}{*}{CFL} & \multirow{2}{*}{Limiter} & \multicolumn{4}{c}{$nx$} \\
\cmidrule(l){3-6}
 & & {33} & {65} & {129} & {257} \\
\midrule
\multirow{4}{*}{0.5}
& Minmod   & 9.20 & 5.90 & 4.01 & 2.71 \\
& MC       & 9.35 & 7.33 & 5.91 & 4.82 \\
& Superbee & 12.35 & 10.35 & 8.17 & 6.40 \\
& Learned  & {\bfseries 7.68} & {\bfseries 4.81} & {\bfseries 3.19} & {\bfseries 2.18} \\
\addlinespace

\multirow{4}{*}{0.25}
& Minmod   & 10.23 & 6.08 & 3.73 & 2.26 \\
& MC       & 7.40 & 4.79 & 3.25 & 2.30 \\
& Superbee & 8.46 & 6.34 & 4.96 & 3.96 \\
& Learned  & {\bfseries 7.29} & {\bfseries 4.26} & {\bfseries 2.64} & {\bfseries 1.66} \\
\addlinespace

\multirow{4}{*}{0.125}
& Minmod   & 11.16 & 6.66 & 4.02 & 2.40 \\
& MC       & 7.31 & 4.33 & 2.64 & 1.63 \\
& Superbee & 7.33 & 4.74 & 3.14 & 2.25 \\
& Learned  & {\bfseries 7.20} & {\bfseries 4.18} & {\bfseries 2.46} & {\bfseries 1.45} \\
\bottomrule
\end{tabular}
\par\smallskip \noindent\parbox{\textwidth}{ \raggedright\footnotesize Bold entries indicate the smallest error among the four limiters for each fixed combination of CFL number and spatial resolution. \par}
\end{table}

Table~\ref{tab:cfl_limiters} compares the learned limiters with the standard Minmod, MC and Superbee limiters for different CFL numbers and mesh sizes. In this experiment, the limiter was trained separately for each discretization setting, so the table does not represent a single universal learned limiter, but rather the best learned limiter associated with each configuration. The learned limiters give the smallest errors for every CFL numbers and meshes and the improvement is most visible on finer meshes. This indicates that the learned approach can effectively adapt to the underlying discretization and reduce the numerical error compared with classical limiters. Moreover, reducing the CFL number does not lead to a monotone improvement for all standard limiters. In contrast, the learned limiters show consistently good performance across the tested CFL values.

\begin{table}[ht!]
\centering
\caption{Experimental comparison of different slope limiters: discrete $L_2$ errors for various initial conditions and final times. Entries are given in units of $10^{-3}$.}
\label{tab:dt_initial_conditions}
\begin{tabular}{cc S[table-format=1.2] S[table-format=1.2] S[table-format=1.2]}
\toprule
\multirow{2}{*}{Initial condition} & \multirow{2}{*}{Limiter} & \multicolumn{3}{c}{$\Delta t$} \\
\cmidrule(l){3-5}
& & {0.25} & {0.5} & {1.0} \\
\midrule
\multirow{4}{*}{Riemann}
& Minmod   & 5.51 & 5.07 & 4.65 \\
& MC       & 4.72 & 5.31 & 6.31 \\
& Superbee & 5.90 & 7.53 & 9.65 \\
& Learned  & {\bfseries 4.65} & {\bfseries 4.53} & {\bfseries 4.31} \\
\addlinespace

\multirow{4}{*}{Linear}
& Minmod   & 2.93 & 3.55 & 3.48 \\
& MC       & 2.74 & 3.72 & 4.25 \\
& Superbee & 3.88 & 5.23 & 6.09 \\
& Learned  & {\bfseries 2.00} & {\bfseries 2.56} & {\bfseries 2.59} \\
\addlinespace

\multirow{4}{*}{Trigonometric}
& Minmod   & 1.05 & 1.35 & 1.12 \\
& MC       & 1.19 & 1.50 & {\bfseries 1.08} \\
& Superbee & 1.76 & 2.02 & 1.26 \\
& Learned  & {\bfseries 0.96} & {\bfseries 1.31} & 1.16 \\
\bottomrule
\end{tabular}
\par\smallskip \noindent\parbox{\textwidth}{ \raggedright\footnotesize Bold entries indicate the smallest error among the four limiters for each fixed combination of initial-condition class and final time. \par}
\end{table}

Table~\ref{tab:dt_initial_conditions} reports the error separately for the three types of initial data and for the three final
times. For the Riemann and linear initial data, the learned limiter achieves the best performance for all final times. The improvement is particularly pronounced for the linear initial data, where the learned reconstruction substantially reduces the error compared with the standard limiters. For the trigonometric initial data, the learned limiter performs best only over the shorter time intervals, while the MC limiter yields the smallest error for $\Delta t = 1.0$. However, the difference in performance is small.

Overall, these tests support two main conclusions. First, the proposed learning approach can outperform classical limiters when the training and testing distributions are consistent. Second, the learned limiter should not be interpreted as a universal replacement for all standard limiters; rather, it is a data-adapted reconstruction strategy whose performance depends on the target data.

Finally, Fig.~\ref{fig_final} illustrates the performance of the complete numerical method with different slope limiters for a Riemann problem. Compared with the classical limiters, the neural network-based limiter provides the most accurate approximation of the exact solution and is significantly less diffusive than the pure Godunov scheme. Its profile lies between those of the minmod and superbee limiters, suggesting that the learned reconstruction achieves a good balance between numerical diffusion and the sharp resolution of the discontinuity. The training was performed using nx $= 65$, cfl$=0.5$, $t_{\textrm{end}}\in\{0.5, 0.25, 0.125\}$, and randomly selected Riemann, linear, and trigonometric initial conditions. Accordingly, in the model problem shown in Fig.~\ref{fig_final}, we used nx $= 65$, cfl$=0.5$, $t_{\textrm{end}} = 0.5.$

\begin{figure}[ht]
    \centering
\includegraphics[width=0.85\linewidth]{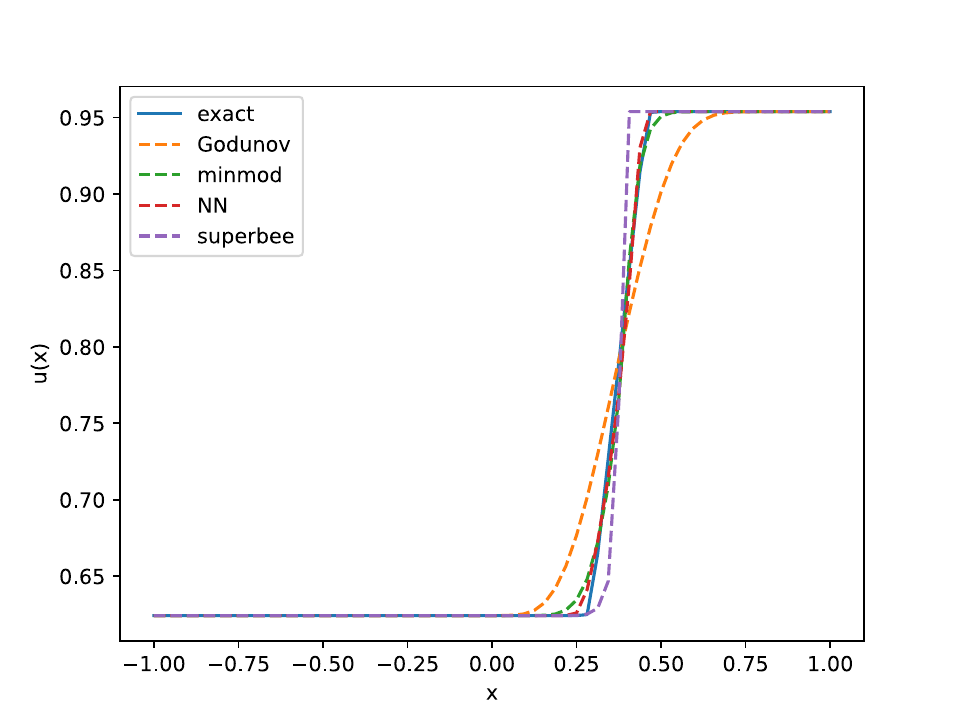}
    \caption{Comparison of the neural network-based limiter with the exact solution, the Godunov scheme, and classical flux limiters for a Riemann problem.}
    \label{fig_final}
\end{figure}

\paragraph{{LWR model}}
We also compared the proposed neural reconstruction approach with the FLOW model from \cite{morand24}. In that approach, the numerical flux $\hat f$ in Eq.~(2) is represented by a dense neural network. The FLOW model was trained for the Lighthill--Whitham--Richards model with the Greenshields flux
$$
f(u)=u-\frac{u^2}{4},
$$
and therefore we used the same flux in the comparison reported in Table~\ref{table:NN}.

\begin{table}[ht!]
\centering
\caption{Discrete $L_{2}$ errors obtained by different neural network-based methods. Entries are given in units of $10^{-3}$.}
\label{table:NN}
\begin{tabular}{cc S[table-format=2.2] S[table-format=2.2] S[table-format=2.2]}
\toprule
\multirow{2}{*}{CFL} & \multirow{2}{*}{Method} & \multicolumn{3}{c}{$nx$} \\
\cmidrule(l){3-5}
 & & {101} & {201} & {401} \\
\midrule
\multirow{3}{*}{0.5}
& CNN\_FLOW   & 17.40 & 11.81 & 8.22 \\
& NN\_R       & {\bfseries 8.06} & 5.87 & 4.60 \\
& NN          & 8.13 & {\bfseries 5.86} & {\bfseries 4.52} \\
\addlinespace
\multirow{3}{*}{0.25}
& CNN\_FLOW   & 20.09 & 13.71 & 9.52 \\
& NN\_R       & 7.72 & 4.81 & 3.15 \\
& NN          & {\bfseries 7.47} & {\bfseries 4.63} & {\bfseries 3.04} \\
\bottomrule
\end{tabular}
\par\smallskip \noindent\parbox{\textwidth}{ \raggedright\footnotesize Bold entries indicate the smallest error among the three neural network-based methods for each fixed combination of CFL number and spatial resolution. \par}
\end{table}

As in the Burgers experiments, we trained on three classes of initial data. Since the FLOW model was trained only on Riemann problems, we considered two versions of our method. The model denoted by NN\_R was trained only on Riemann initial data, whereas NN was trained on all three initial-condition classes. The test measures how the accuracy changes as the mesh is refined.

The results in Table~\ref{table:NN} show that both variants of the proposed method substantially outperform CNN\_FLOW for all tested CFL numbers and spatial resolutions. The difference is already large on the coarsest grid and remains visible as the mesh is refined. Comparing NN\_R and NN also shows that the choice of training data matters. When $\mathrm{CFL}=0.5$, the Riemann-only model is marginally better at the coarsest resolution, but the model trained on the richer data set becomes slightly more accurate on finer grids. For $\mathrm{CFL}=0.25$, the NN model trained on all initial-condition classes is consistently the best.

These experiments provide an additional validation of the proposed approach. Instead of learning the full numerical flux directly, the method learns a reconstruction mechanism embedded into a finite-volume scheme. This keeps the numerical method closer to the structure of classical high-resolution schemes, while still allowing the reconstruction to adapt to the data. The LWR results indicate that this structure-preserving learning strategy can be competitive with, and in the present tests more accurate than, a fully neural flux approximation.


\section{Discussion}\label{discussion}
First, we discuss the complexity of the
neural network approach. Turning to the Godunov type network in Figure \ref{res_fig_full}, we note that in case of $N$ time steps, we use a number of $N\cdot 6$ layers, which
can lead to a really deep neural network. We get even more deep networks by learning
reconstruction schemes.
At the same time, in these cases, the number of parameters, {independently from the number of layers}, is only 4 or 5,
see Table \ref{table_gr} and Table \ref{compare_no_par}.

\begin{table}[H]
\caption{A summary of some neural network methods, architectures, and computational complexity for solving the one-dimensional Burgers equation.}
\label{compare_no_par}
\begin{tabular}
{@{\extracolsep\fill}lccccc}
\toprule%
Ref.& \makecell{Basic\\ approach} &  \multicolumn{3}{c}{Number/size of}  &\\
\cmidrule{3-5}
 & & layers & \makecell{neurons\\per layer} & parameters & Remark\\
 \midrule
\cite{savovic23} & PINN & 3 & 20 & $\sim$1,300 & also: comparison\\
[2pt]
\cite{raissi19} & PINN & 4 & different & $1200$ & \makecell{foundational work} \\
[2pt]
\cite{deryck24} & PINN  & 5 & 50 & $10401$ & \makecell{based on \\weak form}\\
[2pt]
\cite{li21} & FNO & 3 & 40 & $3441$ &  \makecell{foundational work} \\
[2pt]
\cite{berrone23} & FNO & 5 & 20 & $1761$ & \makecell{special treatment\\ for discontinuities} \\
[2pt]
\cite{lu21} & PINN & 4 & 100 & $57888$ & \makecell{deep O-net\\ architecture}\\
[2pt]
\cite{chen22} & NN-flux & 5 & 64 & $17153$ &  \makecell{similar to\\ present}\\
[2pt]
\cite{liu24} & NN-flux & 5 & 64 & $17153$ &  \makecell{entropy-stable \\version of \cite{chen22}}\\
[2pt]
\cite{kim25} & FNO-flux & 5 & 64 & $541$ & \makecell{64 channels,\\ 4 Fourier layers,\\ 16 modes} \\
[2pt]
\makecell{present\\ approach} & \makecell{exact NN\\
Godunov flux} & $64\cdot 6$ & different & $4$ & \makecell{
 including\\ time steps} \\
 [2pt]
 \makecell{present\\ approach} & \makecell{exact NN\\
reconstruction} & $64\cdot 13$ & different & $5$ & \makecell{
 including\\ time steps} \\
\midrule
\end{tabular}
\end{table}





To compare this with other neural networks for this purpose in the literature, we have collected the data of some related approaches. This is shown in Table \ref{compare_no_par}. Note that the number of the parameters are usually not directly comparable, since the corresponding networks are designed for different tasks. One can
observe that the PINN-based methods may use a huge number of parameters.
Whenever FNOs offer function space-based approach, the number of parameters is still high.

In case of works related to our approach, a conservative scheme is assumed and only numerical fluxes are sought using neural networks. Accordingly, they use a moderate number of parameters. In the present work, we could still significantly reduce this number compared to the recent approaches.

 Finally, we have also quantified
 the approximation properties of our neural network-based numerical method in relation { with the conventional approaches using Minmod, Superbee and MC slope limiters.}
 The results presented in Table \ref{tab:dt_initial_conditions} indicate that the proposed method achieves slightly better performance than conventional approaches.

 { At the same time, as the learning procedure yields a standard finite-volume method with optimized slope limiters, the resulting method has essentially the same minimal computational complexity as conventional finite-volume methods. This represents a significant advantage over PINN-based approaches, such as the recent one  \cite{deryck24}, whose resulting models require multiplications with dense matrices  during inference.}

\section{Conclusions and further work}
We have shown here a clear link between conventional numerical methods
and neural network-based approximations for {one-dimensional} scalar conservation laws. The network parameters result in a simple yet efficient
conventional numerical method for solving one-dimensional scalar conservation laws.
This clear and simple setup paves the way for optimizing also the corresponding
time-stepping scheme and extending the approach to multidimensional or vector-valued conservation laws.
{The main steps for vector-valued schemes are similar. For example, instead of Godunov's scheme, the HLL scheme could be reformulated as a parametric neural network. While the application of slope limiters would remain essentially unchanged, constructing a sufficiently rich set of analytically solvable training problems is a non-trivial task for real-world applications such as the compressible Euler equations and the shallow water equations.}

\appendix
\renewcommand{\thesection}{Appendix}

\section{An accurate approximation for learning data}\label{secA1}

    Using the initial condition $u_0(x)=g(x)$, in \eqref{cons_law} with a bounded function $g$, our goal is to approximate the solution at a point $(t,x)\in\mathbb{R}^2$.\\
We will use the method of characteristics. First, assume that they do not intersect up to time $t$, and there exists a unique characteristic passing through $(t,x)$.
Such a characteristic has the slope $u(t,x)$ and the solution remains constant along the curves $$x'(t)=u(t,x(t)).$$
Therefore, it passes the $x$-axis at $x-u(t,x)$, where the solution is
$g(x-u(t,x))$ such that
\begin{equation}\label{impl_char2}
    u(t,x) = g(x - t u(t,x)).
\end{equation}
This holds for all $0 < t < -\frac{1}{\inf_x g'(x)}$, i.e., up to the time characteristics intersect. This is an implicit, nonlinear equation for $u(t,x)$, but due to the boundedness of $g$, the solution is also bounded.

For the numerical solution of \eqref{impl_char2}, we rearrange it as
$$
F(u) = g(x - f'(u)t) - u = 0
$$
and apply Newton's iteration formula
as follows:
$$
u_{n+1} = u_n + \frac{g(x - t u_n) - u_n}{t g'(x - t u_n) + 1}.
$$
Thus, to compute the solution at time $t$, we apply Newton's method for each spatial grid point.

For the specific case:
$$g(x) = r_3 \sin(\pi x + r_2),$$
where $r_3\in[0,1]$ and $r_2\in [0,2\pi]$, we have:
$$g'(x) = \pi r_3 \cos(\pi x + r_2),$$
and therefore,
$$\inf_x g'(x) = -\pi r_3.$$
Hence, characteristics do not intersect when $t \leq \frac{1}{\pi r_3}$, and the above analysis is valid.\\
The range of $g$ is the interval $
R_g = [-r_3, r_3]$,
and Newton's method becomes:
$$
u_{n+1} = u_n + \frac{r_3 \sin(\pi(x - t u_n) + r_2) - u_n}{t \pi r_3 \cos(\pi(x - t u_n) + r_2) + 1}.
$$

In this specific initial condition, the exact solution can be obtained even for fixed $t > \frac{1}{\pi r_3}$. First note that the initial condition is a full sine wave over the interval $[-1,1]$, and the solution remains symmetric with respect to the point $(xs, 0)$, where
$$
xs = \frac{\pi - r_2}{\pi}.
$$
Hence, a shock develops only at $xs$, and remains stationary (moving with $0$ velocity) due to the Rankine--Hugoniot condition. Therefore, the earlier method remains applicable for any $t > 0$.\\
A critical aspect of Newton's method is choosing an appropriate initial guess. Therefore the order in which we compute the solution values is really important. First, the boundary conditions are computed for all time steps $k\delta,\ (k=1,\ldots,K)$ using Newton's method. For small $\delta$, the boundary values change only a little between time steps, making the previous value a good initial guess. At $t=0$, the solution is known at $x=-1$ and $x=1$, enabling recursive computation of boundary values.\\
Subsequently, the solution at $t=K \times \delta$ is determined. Since the solution is known at the boundaries and is continuous except at $xs$, we can proceed inward from both boundaries using Newton's method, initializing with the value from the neighboring point.

\section*{Funding and Acknowledgements}
This research was supported by the Ministry of Innovation and Technology NRDI Office within the framework of the Artificial Intelligence National Laboratory Program RRF-2.3.1-21-2022-00004, and within the framework of the Thematic Excellence Program ELTE TKP 2021-NKTA-62.

\medskip
\noindent We acknowledge the Digital Government Development and Project Management Ltd. for awarding us access to the Komondor HPC facility based in Hungary.

\bibliographystyle{plain}
\bibliography{bib_JSC}

@article {hillebrand23,
    AUTHOR = {Hillebrand, Dorian and Klein, Simon-Christian and \"Offner,
              Philipp},
     TITLE = {Applications of limiters, neural networks and polynomial
              annihilation in higher-order {FD}/{FV} schemes},
   JOURNAL = {J. Sci. Comput.},
  FJOURNAL = {Journal of Scientific Computing},
    VOLUME = {97},
      YEAR = {2023},
    NUMBER = {1},
     PAGES = {Paper No. 13, 31},
      ISSN = {0885-7474,1573-7691},
   MRCLASS = {65M08 (65M06)},
  MRNUMBER = {4637483},
}

@article{morand24,
  title     = {Deep learning of first-order nonlinear hyperbolic conservation law solvers},
  author    = {Morand, Victor and M{\"u}ller, Nils and Weightman, Ryan and Piccoli, Benedetto and Keimer, Alexander and Bayen, Alexandre M.},
  journal   = {Journal of Computational Physics},
  volume    = {511},
  pages     = {113114},
  year      = {2024},
  publisher = {Elsevier},
  }

@article{nguyen22,
    author = {Nguyen-Fotiadis, Nga and McKerns, Michael and Sornborger, Andrew},
    title = {Machine learning changes the rules for flux limiters},
    journal = {Physics of Fluids},
    volume = {34},
    number = {8},
    pages = {085136},
    year = {2022},
    month = {08},
}

@article{nogueira24,
author = {Xesús Nogueira and Javier Fernández-Fidalgo and Lucía Ramos and Iván Couceiro and Luis Ramírez},
title = {Machine learning-based {WENO5} scheme},
journal = {Computers $\&$ Mathematics with Applications},
volume = {168},
pages = {84-99},
year = {2024},
issn = {0898-1221},
}

@article{kossaczka24,
    author = {Kossaczká, Tatiana and Jagtap, Ameya D. and Ehrhardt, Matthias},
    title = {Deep smoothness weighted essentially non-oscillatory method for two-dimensional hyperbolic conservation laws: A deep learning approach for learning smoothness indicators},
    journal = {Physics of Fluids},
    volume = {36},
    number = {3},
    pages = {036603},
    year = {2024},
    month = {03},
}

@article{ruggeri22,
    author = {Ruggeri, Matteo and Roy, Indradip and Mueterthies, Michael J. and Gruenwald, Tom and Scalo, Carlo},
    title = {Neural-network-based {R}iemann solver for real fluids and high explosives; application to computational fluid dynamics},
    journal = {Physics of Fluids},
    volume = {34},
    number = {11},
    pages = {116121},
    year = {2022},
    month = {11},
}

@article{xu25,
author = {Liang Xu and Ziyan Liu and Yiwei Feng and Tiegang Liu},
title = {Unsupervised neural-network solvers for multi-material {R}iemann problems},
journal = {Computer Physics Communications},
volume = {308},
pages = {109470},
year = {2025},
issn = {0010-4655},
}

@article{bardos79,
  author    = {Claude Bardos and Alain Y. le Roux and Jean-Claude Nédélec},
  title     = {First-order quasilinear equations with boundary conditions},
  journal   = {Communications in Partial Differential Equations},
  volume    = {4},
  number    = {9},
  pages     = {1017--1034},
  year      = {1979},
  publisher = {Taylor \& Francis},
}

@article{godunov59,
  author    = {S. K. Godunov},
  title     = {A Difference Method for Numerical Calculation of Discontinuous Solutions of the Equations of Hydrodynamics},
  journal   = {Matematicheskii Sbornik},
  year      = {1959},
  volume    = {47},
  number    = {3},
  pages     = {271--306},
  note      = {English translation in: J. Comput. Phys., 1969},
}

@article{harten83,
  author    = {Ami Harten},
  title     = {High Resolution Schemes for Hyperbolic Conservation Laws},
  journal   = {Journal of Computational Physics},
  year      = {1983},
  volume    = {49},
  number    = {3},
  pages     = {357--393},
}

@book{leveque92,
  author    = {Randall J. LeVeque},
  title     = {Numerical Methods for Conservation Laws},
  publisher = {Birkh\"auser},
  year      = {1992},
  edition   = {2nd},
  address   = {Basel},
  isbn      = {9783764327231}
}

@article{liu94,
author = {Xu-Dong Liu and Stanley Osher and Tony Chan},
title = {Weighted Essentially Non-oscillatory Schemes},
journal = {Journal of Computational Physics},
volume = {115},
number = {1},
pages = {200-212},
year = {1994},
issn = {0021-9991},
}

@article{jiang96,
author = {Guang-Shan Jiang and Chi-Wang Shu},
title = {Efficient Implementation of Weighted {ENO} Schemes},
journal = {Journal of Computational Physics},
volume = {126},
number = {1},
pages = {202-228},
year = {1996},
issn = {0021-9991},
}

@article{chen24,
author = {Chen, Zhen and Gelb, Anne and Lee, Yoonsang},
title = {Learning the Dynamics for Unknown Hyperbolic Conservation Laws Using Deep Neural Networks},
journal = {SIAM Journal on Scientific Computing},
volume = {46},
number = {2},
pages = {A825-A850},
year = {2024},
}

@article{lighthill55,
  title={On kinematic waves. II. A theory of traffic flow on long crowded roads},
  author={Lighthill, M. J. and Whitham, G. B.},
  journal={Proceedings of the Royal Society of London. Series A, Mathematical and Physical Sciences},
  volume={229},
  number={1178},
  pages={317--345},
  year={1955},
  publisher={The Royal Society}
}

@article{richards56,
  title={Shock waves on the highway},
  author={Richards, Paul I.},
  journal={Operations Research},
  volume={4},
  number={1},
  pages={42--51},
  year={1956},
  publisher={INFORMS}
}

@article{kim25,
  author  = {Kim, Taeyoung and Kang, Myungjoo},
  title   = {Approximating Numerical Fluxes Using {F}ourier Neural Operators for Hyperbolic Conservation Laws},
  journal = {Communications in Computational Physics},
  volume  = {37},
  number  = {2},
  pages   = {420--456},
  year    = {2025},
}

@article{chen22,
  title={Designing Neural Networks for Hyperbolic Conservation Laws},
  author={Chen, Zhen and Gelb, Anne and Lee, Yoonsang},
  journal={arXiv preprint arXiv:2211.14375},
  year={2022},
}

@article{raissi19,
  title   = {Physics-informed neural networks: A deep learning framework for solving forward and inverse problems involving nonlinear partial differential equations},
  author  = {Raissi, Maziar and Perdikaris, Paris and Karniadakis, George Em},
  journal = {Journal of Computational Physics},
  volume  = {378},
  pages   = {686--707},
  year    = {2019},
}

@inproceedings{li21,
  title     = {Fourier Neural Operator for Parametric Partial Differential Equations},
  author    = {Li, Zongyi and Kovachki, Nikola B. and Azizzadenesheli, Kamyar and Liu, Burigede and Bhattacharya, Kaushik and Stuart, Andrew M. and Anandkumar, Anima},
  booktitle = {International Conference on Learning Representations (ICLR)},
  year      = {2021},
}

@article{deryck24,
  title={{wPINNs}: Weak Physics Informed Neural Networks for Approximating Entropy Solutions of Hyperbolic Conservation Laws},
  author={De Ryck, Tim and Mishra, Siddhartha and Molinaro, Roberto},
  journal={SIAM Journal on Numerical Analysis},
  volume={62},
  number={2},
  pages={569--600},
  year={2024},
  publisher={SIAM}
}

@article{berrone23,
  title={Physics-informed neural networks with special treatment of discontinuities},
  author={Berrone, Stefano and D'Amato, Andrea},
  journal={Journal of Computational and Applied Mathematics},
  volume={429},
  pages={115162},
  year={2023},
  publisher={Elsevier}
}

@article{lu21,
  title={{Learning nonlinear operators via DeepONet based on the universal approximation theorem of operators}},
  author={Lu, Lu and Jin, Pengzhan and Pang, Guofei and Zhang, Ziqiang and Karniadakis, George E},
  journal={Nature Machine Intelligence},
  volume={3},
  number={5},
  pages={453--458},
  year={2021},
  publisher={Nature Publishing Group}
}

@Article{savovic23,
AUTHOR = {Savović, Svetislav and Ivanović, Miloš and Min, Rui},
TITLE = {A Comparative Study of the Explicit Finite Difference Method and Physics-Informed Neural Networks for Solving the {B}urgers’ Equation},
JOURNAL = {Axioms},
VOLUME = {12},
YEAR = {2023},
NUMBER = {10},
ARTICLE-NUMBER = {982},
}

@article{liu24,
  title={Entropy-stable conservative flux form neural networks},
  author={Liu, Lizuo and Li, Tongtong and Gelb, Anne and Lee, Yoonsang},
  year={2024},
note={available at: \url{https://arxiv.org/abs/2411.01746}},
}

@article{kovachki23,
  title        = {Neural Operator: Learning Maps Between Function Spaces With Applications to PDEs},
  author       = {Nikola Kovachki and Zongyi Li and Burigede Liu and Kamyar Azizzadenesheli and Kaushik Bhattacharya and Andrew Stuart and Anima Anandkumar},
  journal      = {Journal of Machine Learning Research},
  volume       = {24},
  number       = {89},
  pages        = {1--97},
  year         = {2023},
}

@techreport{berger05,
  author       = {Marsha Berger and Michael J. Aftosmis and Scott M. Murman},
  title        = {Analysis of Slope Limiters on Irregular Grids},
  institution  = {NASA Ames Research Center},
  year         = {2005},
  number       = {NAS-05-007},
  address      = {Moffett Field, CA},
}
\end{document}